%% file: strong_conv_sa_nonconvex.tex
\RequirePackage{fix-cm}
\documentclass[smallextended]{svjour3}       % onecolumn (second format)

\smartqed  % flush right qed marks, e.g. at end of proof

\usepackage{graphicx}

\usepackage{subfig}
\usepackage[leqno]{amsmath}
\allowdisplaybreaks

\usepackage{amsfonts}
\usepackage{amssymb}
\usepackage{hyperref}[6.83]
\hypersetup{colorlinks,
			linkcolor=[rgb]{.61,0,0.3},
			citecolor=[rgb]{.14,.47,.14}}
\usepackage{enumitem}
\usepackage{longtable}
\usepackage{booktabs}

\usepackage{nomencl}
\usepackage{etoolbox}

\patchcmd{\thenomenclature}{\section*}{\section}{}{}

\makenomenclature

\usepackage{xstring}
\usepackage{xpatch}

\newcommand{\nomenclheader}[1]{%
  \item[\hspace*{-\itemindent}\normalfont\itshape#1]}
\renewcommand\nomgroup[1]{%
  \bigskip
  \IfStrEqCase{#1}{%
   {A}{\nomenclheader{Numerical Analysis}}%      
   {B}{\nomenclheader{Probability}}% 
   {C}{\nomenclheader{Stochastic Gradient Method and Stopping Criteria}}% 
   {D}{\nomenclheader{Conditions on Nonconvexity}}%
  }%
  
}

\usepackage{marginnote}
\usepackage{lineno}
\newcommand{\norm}[1]{\left\Vert #1 \right\Vert}
\newcommand{\1}[1]{\textbf{1}\left[ #1 \right]}
\newcommand{\Prb}[1]{\mathbb{P}\left[ #1 \right]}
\newcommand{\E}[1]{\mathbb{E}\left[ #1 \right]}

\newcommand{\cond}[2]{\mathbb{E}\left[\left. #1 \right\vert #2 \right]}
\newcommand{\condPrb}[2]{\mathbb{P}\left[\left. #1 \right\vert #2 \right]}

\makeatletter
\newcommand{\leqnomode}{\tagsleft@true\let\veqno\@@leqno}
\newcommand{\reqnomode}{\tagsleft@false\let\veqno\@@eqno}
\makeatother

\newcommand{\condition}[3]{
\begin{center}
\noindent\fbox{
\parbox{\dimexpr\textwidth-3em}{
\leqnomode
\begin{equation}\label{#1}
\parbox{\dimexpr\linewidth-5.5em}{#2}\tag{\textbf{#3}}
\end{equation}}
}
\reqnomode
\end{center}
}
\begin{document}
\reqnomode
\title{Stopping Criteria for, and Strong Convergence of, Stochastic Gradient Descent on Bottou-Curtis-Nocedal Functions
\thanks{This work is supported by the Wisconsin Alumni Research Foundation.}
}
\subtitle{}

\titlerunning{Stopping and Strong Convergence}        % if too long for running head

\author{Vivak Patel}


\institute{Vivak Patel \at
              Department of Statistics \\
              University of Wisconsin -- Madison \\
              \email{vivak.patel@wisc.edu}}


\maketitle

\begin{abstract}
Stopping criteria for Stochastic Gradient Descent (SGD) methods play important roles from enabling adaptive step size schemes to providing rigor for downstream analyses such as asymptotic inference. Unfortunately, current stopping criteria for SGD methods are often heuristics that rely on asymptotic normality results or convergence to stationary distributions, which may fail to exist for nonconvex functions and, thereby, limit the applicability of such stopping criteria. To address this issue, in this work, we rigorously develop two stopping criteria for SGD that can be applied to a broad class of nonconvex functions, which we term Bottou-Curtis-Nocedal functions. Moreover, as a prerequisite for developing these stopping criteria, we prove that the gradient function evaluated at SGD's iterates converges strongly to zero for Bottou-Curtis-Nocedal functions, which addresses an open question in the SGD literature. As a result of our work, our rigorously developed stopping criteria can be used to develop new adaptive step size schemes or bolster other downstream analyses for nonconvex functions. 

\end{abstract}

\section{Introduction} \label{section-introduction}
\input{section/introduction.tex}

\input{section/notation.tex}

\section{Challenges of Rigorous Stopping Criteria} \label{section-rigorous}
\input{section/rigorous}

\section{Bottou-Curtis-Nocedal and Other Conditions on Nonconvexity} \label{section-structure}
\input{section/literature.tex}

\section{Strong Global Convergence} \label{section-stability-convergence}
\input{section/intro_convergence.tex}
\input{section/method.tex}
\input{section/global_convergence.tex}

\section{Stopping Criteria} \label{section-stopping}
\input{section/intro_stopping_criteria.tex}
\input{section/statement_stopping_criteria.tex}
\input{section/statement_scenarios.tex}
\input{section/sc_direct_estimate.tex}

\input{section/sc_majority_vote.tex}

\section{Numerical Experiment} \label{section-experiment}
\input{section/experiment}

\section{Conclusion} \label{section-conclusion}
\input{section/conclusion.tex}

\begin{acknowledgements}
We thank the reviewers for their detailed feedback, which has greatly improved the quality of this work.
\end{acknowledgements}


\bibliographystyle{spmpsci}      % mathematics and physical sciences
\bibliography{strong_conv_sa.bib}   % name your BibTeX data base

\end{document}

%% file: section/introduction.tex
In data-driven and simulation-based disciplines, the optimization problem
\begin{equation} \label{problem}
\min_{\theta} F(\theta)
\end{equation}
is frequently solved, 
where $F(\theta) = \E{ f(\theta,X)}$; $f: \mathbb{R}^p \times \mathbb{R}^d \to \mathbb{R}$; $X$ is a random variable with arbitrary support; and $\mathbb{E}$ is the expectation operator. Depending on $X$, the optimization problem's objective function and its gradient may be impractical or impossible to evaluate directly \cite{bertsekas2011}. Fortunately, when certain regularity conditions hold on $f$ and $X$ (e.g., \cite{bottou2018,wu2009}), the optimization problem's structure is exploited to generate solvers that use the gradient of $f$ with respect to the argument $\theta$ for independent copies of $X$ \cite{robbins1951,chung1954}. Moreover, when the gradient of $f$ is significantly cheaper to compute than the gradient of $F$, the optimization problem can be efficiently solved using these so-called Stochastic Gradient Descent (SGD) methods \cite{bertsekas2011}.

Despite the potential efficiency of SGD methods, they are fickle algorithms as they are highly sensitive to the choice of step sizes \cite{nemirovski2009,patel2017}; and, consequently, SGD robustification using adaptive stepping procedures has been an active area of research since the inception of stochastic approximation procedures \cite{kiefer1952,farrell1962,fabian1967,sielken1973,stroup1982,ermoliev1983,mirozahmedov1983}.\footnote{For a review of recent approaches to adaptive step size procedures, see \cite{curtis2020}.} Critically, SGD methods with adaptive stepping procedures lay out \textit{criteria} for when to \textit{stop} the current step size schedule, and then provide a procedure for how to modify it. Thus, SGD's robustification is inherently tied to stopping criteria (see \S 17 of \cite{pflug1988}).

While their utility in the adaptive step size selection for SGD methods cannot be overstated, stopping criteria play other, equally important roles. In the deterministic setting, stopping criteria determine whether asymptotic inference tools can be applied (e.g., Theorem 5.23 in \cite{vandervaart2000}), serve as a proxy for generalization quality \cite{prechelt1998},\footnote{There has been some recent work that disagrees with whether early termination leads to better generalization. See \cite{bassily2018} and related work. However, even in this case, one needs to know that a minimizer is achieved.} and indicate when to terminate an algorithm. Even in the stochastic setting, these applications of stopping criteria are essential and have warranted investigation. 

Historically and even recently, stopping criteria---either studied independently or within the context of adaptive stepsize rules---have focused almost entirely on either determining when a confidence region of the iterates contains a strict minimizer \cite{sielken1973,stroup1982,yin1990,wada2010,patel2016}, or have focused on determining when the iterates satisfy a notion of statistical stationarity \cite{mirozahmedov1983,chee2018,zhang2020}. Unfortunately, such stopping criteria are often limited on modern problems. For example, stopping criteria that rely on finding confidence regions that contain a strict minimizer are often limited to strongly convex problems,\footnote{We can include those problems that are nonconvex, yet are locally strongly convex around minimizers.} and are not applicable more broadly, say, to the increasingly common class of overparametrized, nonconvex optimization problems found in machine learning \cite{bassily2018}. Moreover, such confidence-region stopping criteria have had limited rigorous development except in special cases \cite{pflug1988,wada2010,patel2016}.

Thus, stopping criteria that rely on statistical stationarity would seem preferable, yet these criteria face two difficulties. First, such stopping criteria are generally only employable with sufficiently small step sizes (e.g., \cite{zhang2020}), where the degree of smallness has yet to be clearly specified. Additionally, such stopping criteria are rather brittle for general nonconvex problems for two reasons: (a) a stationary distribution need not exist, especially with degenercy of the solution manifold; and (b) such stopping criteria make use of only a single iterate sequence (i.e., a single ``Markov chain'') from which to determine statistical stationarity, which is highly unreliable based on more rigorous stationarity detection methods used in Markov Chain Monte Carlo methods \cite{roy2020}. Unfortunately, just as for confidence-region stopping criteria, these statistical stationarity stopping criteria for stochastic optimization have had limited rigorous development, excepting special cases \cite{mirozahmedov1983,pflug1988,chee2018}.

To begin addressing the aforementioned limitations, we expand the scope of stopping criteria for stochastic optimization by rigorously developing analogues to gradient-based stopping criteria used in deterministic optimization. By taking this gradient based perspective to stopping criteria,
\begin{enumerate}
\item we are able to supply stopping criteria that are applicable to a rather broad set of nonconvex problems, which already goes beyond the scope of confidence-region and statistical stationarity stopping criteria; and,
\item by focusing on a rigorous understanding of our stopping criteria, we demonstrate that the gradient of $F$ evaluated at the iterates converges to zero with probability one, which improves the analysis of stochastic gradient methods for a broad class of nonconvex problems and noise models.\footnote{Including a number of reports that came out after a preprint of this work was made public. See \cite{gower2020,khaled2020,mertikopoulos2020}.}
\end{enumerate}
As a result, our stopping criteria lay a rigorous foundation for developing adaptive step size algorithms, and our stopping criteria can be leveraged for other downstream analyses such as developing statistical uncertainty sets.

The remainder of this paper is organized as follows. In \S2, we provide an overview of important notation. In \S\ref{section-rigorous}, we detail the challenges of rigorously developing stopping criteria, and how our work contributes to the analysis of stochastic gradient descent on nonconvex functions. In \S\ref{section-structure}, we review, catalog, and organize common conditions on $F$ that specify its nonconvexity; and, in \S\ref{subsection-bcn-specification}, we specify a general class of nonconvex functions which we refer to as Bottou-Curtis-Nocedal functions \cite{bottou2018}. In \S\ref{section-stability-convergence}, we precisely specify the stochastic gradient descent iterates (\S \ref{subsection-strong-convergence}), and prove that the gradient evaluated at the iterates converge strongly to zero. In \S \ref{section-stopping}, we introduce and analyze our two stopping criteria. In \S \ref{section-experiment}, we examine the behavior of our stopping criteria on a real data analysis task. In \S \ref{section-conclusion}, we conclude this work.

%% file: section/notation.tex
\label{section-notation}

\renewcommand{\nompreamble}{Below is a list of notation that is used throughout the text. The notation is organized into several groups: notation related to numerical analysis; notation related to probability; notation related the stochastic gradient method and stopping criteria; notation related to conditions placed on the objective function.}

\nomenclature[A, 1]{$F$}{Objective function that maps $\mathbb{R}^p \to \mathbb{R}$. Also referred to as ``Deterministic Objective.''}
\nomenclature[A, 2]{$\dot F$}{Gradient function of the deterministic objective that maps $\mathbb{R}^p \to \mathbb{R}^p$. Also referred to as ``Deterministic Gradient.''}
\nomenclature[A, 3]{$f$}{A function that maps $\mathbb{R}^p \times \mathbb{R}^d \to \mathbb{R}$. Used to construct the deterministic objective, and is referred to as the ``Stochastic Objective.''}
\nomenclature[A, 4]{$\dot f$}{Gradient function of the stochastic objective with respect to the first argument, and maps $\mathbb{R}^p \times \mathbb{R}^d \to \mathbb{R}^p$. Also referred to as the ``Stochastic Gradient.''}
\nomenclature[A, 5]{$\theta$}{The parameter to be optimized in the objective function. Takes value in $\mathbb{R}^p$.}
\nomenclature[A, 6]{$\norm{\cdot}_2$}{If the argument is a vector, this refers to the Euclidean norm of the vector. If the argument is a matrix, this refers to the Operator 2-norm.}
\nomenclature[A, 7]{$\lambda_{\max}$}{The largest eigenvalue corresponding to a given symmetric, positive definite matrix argument.}
\nomenclature[A, 8]{$\lambda_{\min}$}{The smallest eigenvalue corresponding to the given symmetric, positive definite matrix argument.}
\nomenclature[A, 9]{$\kappa $}{The 2-norm condition number of the given matrix.}

\nomenclature[B, 1]{$\mathbb{P}$}{The probability operator with respect to the underlying Probability Space.}
\nomenclature[B, 2]{$\mathbb{E}$}{The expectation operator corresponding to the underlying Probability Space.}
\nomenclature[B, 3]{$X, X_k$}{Independent, identically distributed random variables that map from the underlying Probability Space to $\mathbb{R}^d$.}
\nomenclature[B, 4]{$\tau, T_j$}{Stopping times defined with respect to an appropriate filtration.}
\nomenclature[B, 5]{$\mathcal{F}_k$}{$\sigma$-algebras. $\lbrace \mathcal{F}_k \rbrace$ denotes the filtration generated by $\lbrace X_k \rbrace$.}
\nomenclature[B, 6]{$\Delta$}{Symmetric set difference operator.}
\nomenclature[B, 7]{$i.o.$}{Infinitely often.}
\nomenclature[B, 8]{$\mathbf{1}$}{The indicator function over a specified event.}

\nomenclature[C, 1]{$\beta_k$}{The iterates generated by Stochastic Gradient Descent (SGD).}
\nomenclature[C, 2]{$M_k$}{The matrix-valued learning rate used to generate the SGD iterates.}
\nomenclature[C, 3]{$K$}{A deterministic integer defined in Lemma \ref{lemma-minimality-smallest-eigenvalue}}
\nomenclature[C, 4]{$\psi_k$}{A resumed sequence of iterates of SGD defined with respect to a given finite stopping time.}
\nomenclature[C, 5]{$P_k$}{The resumed matrix-valued learning rate of SGD defined with respect to a given finite stopping time.}
\nomenclature[C, 6]{$Z_k$}{A resumed sequence of random variables, $\lbrace X_k \rbrace$, defined with respect to a given finite stopping time.}
\nomenclature[C, 7]{$\mathcal{G}_k$}{A resumed $\sigma$-algebra defined with respect to a stopped $\sigma$-algebra.}
\nomenclature[C, 8]{$S$}{A constant used to denote an upper bound for the sum of $\lambda_{\max}(M_k)^2$.}
\nomenclature[C, 90]{P}{A designation referring to the four properties (\ref{prop-sympd}, \ref{prop-maxeig}, \ref{prop-mineig}, \ref{prop-condnum}) required of SGD.}
\nomenclature[C, 91]{SC}{A designation referring to a stopping criterion. The abbreviation refers to ``Stopping Criterion.''}

\nomenclature[D, 01]{$F_{l.b.}$}{The lower bound value for the objective function.}
\nomenclature[D, 020]{$C_{\cdot}$}{Constants used in specifying assumptions about the optimization problem.}
\nomenclature[D, 021]{$\pi_{\cdot}$}{Constants used in specify the assumptions for tail probabilities.}
\nomenclature[D, 03]{IS}{A designation referring to assumptions about the domain of the objective function. The abbreviation stands for ``Iterate Space.''}
\nomenclature[D, 04]{DO}{A designation referring to assumptions about the deterministic objective function. The abbreviation stands for ``Deterministic Objective.''}
\nomenclature[D, 05]{SO}{A designation referring to assumptions about the stochastic objective function. The abbreviation stands for ``Stochastic Objective.''}
\nomenclature[D, 06]{DG}{A designation referring to assumptions about the deterministic gradient function. The abbreviation stands for ``Deterministic Gradient.''}
\nomenclature[D, 07]{SG}{A designation referring to assumptions about the stochastic gradient function. The abbreviation stands for ``Stochastic Gradient.''}
\nomenclature[D, 08]{NM}{A designation referring to assumptions about the variability of the stochastic gradient function. The abbreviation refers to ``Noise Model.''}
\nomenclature[D, 09]{T}{A designation referring to assumptions about the tail behavior of the stochastic gradient function. The abbreviation stands for ``Tail.''}
\nomenclature[D, 10]{BCN}{A reference to a set of assumptions that define a class of functions which we refer to as Bottou-Curtis-Nocedal functions. See \S \ref{subsection-bcn-specification}.}

\printnomenclature

%% file: section/rigorous.tex
Fundamentally, stopping criteria are posing a zero-one question---that is, either some event does not happen (i.e., zero) or the event does happen (i.e., one). In a deterministic setting, useful stopping criteria are typically asked about observable quantities, such as the value of the objective function or the norm of the gradient function. Additionally, in the deterministic setting, rigorous stopping criteria theory requires ensuring that the stopping criterion will be triggered, which, in turn, relies on proving that the underlying optimization procedure generates iterates that converge to a stationary point or, more broadly, that induce a reduction in the objective.

In the stochastic setting, these useful deterministic stopping criteria cannot be directly applied as the quantities on which they are based (e.g., the objective function or gradient function) cannot be feasibly evaluated. Thus, in the stochastic setting, such deterministic stopping criteria can be adapted by either estimating the deterministic objective or deterministic gradient function (e.g., \ref{stop-mean-est}), or estimating the zero-one question posed by the deterministic stopping criteria (e.g., \ref{stop-vote-ind}). 

Unfortunately, these estimated stopping criteria inherently require a more complex analysis. Specifically, these estimated stopping criteria require
\begin{enumerate}
\item establishing that the underlying deterministic stopping criterion will be triggered;
\item establishing that the estimated stopping criterion will be triggered;
\item establishing that, if the estimated stopping criterion is triggered, whether the underlying deterministic stopping criterion would be triggered if the quantities involved were computable (i.e., false positive control); and
\item establishing that, if the underlying deterministic stopping criterion would be triggered, whether the estimated stopping criterion is triggered (i.e., false negative control).
\end{enumerate}
We now explore each of these questions below, and we will focus our discussion around the common deterministic stopping criterion that is triggered when the deterministic gradient is below a given threshold. Moreover, when appropriate, we mention how we contribute to the broader literature through our analysis.

\subsection{Triggering the Underlying Deterministic Stopping Criterion}

The usual, deterministic gradient-based stopping criterion requires evaluating whether the norm of the gradient is less than some threshold. In the deterministic context, such a stopping criterion is guaranteed to be triggered if we can prove that the gradient function evaluated at the iterates is converging to zero. In the stochastic setting, any estimated stopping criterion for this underlying deterministic criterion requires establishing the same result, but with the added consideration of several distinct modes of convergence (e.g., strong convergence, converge in probability, converge in distribution). As we now argue, in the stochastic setting, estimated stopping criteria require the high standard of strong convergence (i.e., almost sure convergence, with probability one convergence).

To illustrate why we need strong convergence, consider stopping a sequence of $\lbrace 0, 1 \rbrace$-valued, independent random variables when we observe a sequence of, say, twenty, consecutive zeros with (possibly) random evaluation points, $\lbrace T_j \rbrace \subset \mathbb{N}$. 
Specifically, let $\lbrace X_k \rbrace$ be independent random variables such that $\mathbb P [ X_k = 0 ] = 1 - k^{-\alpha}$ for some $\alpha \in (0,1]$.
Then, $\lbrace X_k \rbrace$ is converging to zero in probability, yet, by the second Borel-Cantelli lemma, $\mathbb P [ X_k = 1, ~i.o. ] = 1$, where $i.o.$ means infinitely often. 
As a result, it is entirely possible that a one appears in $\lbrace X_{T_j},X_{T_j+1},\ldots,X_{T_j+19} \rbrace$ for all $j$; that is, because strong convergence does not hold, a poor choice of $\lbrace T_j \rbrace$ would prevent us from ever observing a consecutive sequence of twenty zeros. Thus, stopping criteria are predicated on establishing strong convergence, rendering convergence in probability insufficient.

Unfortunately, demonstrating strong convergence has been achieved with varying success. When $F$ is convex, strong convergence can be readily established along with other convergence rate results \cite{chung1954,bertsekas2011,bottou2018}. Curiously, even when $F$ is convex, strong convergence has not translated to useful stopping criteria with the exception of the limited stopping criteria demonstrated for quadratic problems (e.g., \cite{mirozahmedov1983,pflug1988,wada2010,patel2016,chee2018}).

When $F$ is nonconvex, convergence in probability is already sufficiently challenging, even under the myriad of notions of nonconvexity that we catalog in \S \ref{section-structure}. In fact, consider a popular, general class of nonconvex functions for which (1) $F(\theta)$ is bounded from below; (2) $\dot f(\theta,X)$ are unbiased stochastic gradients; (3) $\dot{F}(\theta)$, the deterministic gradient, is globally Lipschitz continuous; and (4) $\mathbb{E}[\Vert \dot f (\theta,X) \Vert_2^2 ]$, the noise model, is controlled by a constant summed with a scaling of the norm of the deterministic gradient squared. For such nonconvex functions---which we refer to as Bottou-Curtis-Nocedal (\ref{bcn}) functions, as popularized by \cite{bottou2018} and are formally specified in \S \ref{subsection-bcn-specification}---even convergence in probability has yet to be established. Despite this, there are results under more stringent conditions: 
\begin{enumerate}
\item If a BCN function is also twice differentiable and its Hessian-Gradient product is Lipschitz continuous, then the convergence of the deterministic gradient (evaluated at the iterates of the SGD method) to zero in probability has been demonstrated (see Corollary 4.12 of \cite{bottou2018}).
\item If $f \geq 0$ with probability one and $\dot{f}$ is \textit{globally and uniformly} Lipschitz continuous in $\theta$, then the convergence of the deterministic gradient to zero in probability has also been demonstrated (see Theorem 2(c) of \cite{lei2019}).
\item If a BCN function also requires $
\mathbb{P}[\forall \theta,~ \Vert \dot f(\theta,X) - \dot{F}(\theta)\Vert_2 \leq C] = 1$,
and $F$ to be Lipschitz continuous,
then the strong convergence of the deterministic gradient to zero has been demonstrated (see Theorem 1 of \cite{liorab2018}).
\end{enumerate}
This ultimate result is a rather important contribution, but the additional conditions exclude the simple linear regression problem, which is included in BCN functions. In summary, strong convergence for general, BCN nonconvex functions has yet to be established.

In order to address this gap, our first contribution is to prove the strong convergence (i.e., convergence with probability one) of $\dot{F}$ evaluated at the iterates of SGD to zero for \ref{bcn} functions (see Corollary \ref{corollary-convergence}), which generalizes and strengthens the preceding results. Our proof employs several strategies that either refine current techniques or are atypical in the stochastic optimization literature; our general strategy and how it is distinct from \cite{lei2019} and \cite{liorab2018} is discussed at the beginning of \S \ref{subsection-strong-convergence}. Owing to this contribution, we can guarantee that the underlying deterministic stopping criterion will be triggered.

\subsection{Triggering the Estimated Stopping Criterion}

Unfortunately, even if the underlying deterministic stopping criterion is triggered, the estimated stopping criterion is not necessarily subject to such a guarantee. To illustrate, consider applying the independent-vote stopping criterion (see \ref{stop-vote-ind} below)---that is, stop if
\begin{equation}
\frac{1}{N} \sum_{i=1}^N \1{ | \dot f (\theta,X_i)| \leq \epsilon } \geq \gamma,
\end{equation}
where $\dot{f}(\theta,X_i)$ is the derivative of $f$ with respect to $\theta$; $X,X_i$ are independent and identically distributed; $\gamma,\epsilon \in (0,1)$---to $f(\theta,X) = \theta X$ where $\theta \in \mathbb{R}$ and $X$ is a Rademacher random variable. Then, $\dot{F}(\theta) = 0$ for all $\theta$, yet the stopping criterion will never be triggered since $|\dot{f}(\theta,X_i)| = 1$. 

As this example illustrates, the vote-based stopping criterion faces a challenge because the stochastic gradient's distribution is concentrated away from its mean. If such a noise condition on the stochastic gradients is endogenous to the optimization problem, the vote-based stopping criterion would be a poor choice and an alternative would be more appropriate. On the other hand, if such a noise condition is not present in the optimization problem, the vote-based stopping criterion would offer certain computational and information-theoretic benefits that we discuss further below.

To summarize, whether an estimated stopping criterion is triggered is \textit{not} subsumed by whether the underlying deterministic stopping criterion is triggered, and the choice of estimated stopping criterion will have implications for performance even for the same underlying deterministic stopping criterion. Thus, studying whether the estimated stopping criterion is triggered adds a layer of complexity that does not exist for the deterministic context. To our knowledge, such an investigation has not been undertaken in the literature, and will be a central issue that we will explore.

\subsection{False Negative and False Positive Control}
Owing to the fact that estimated stopping criteria depend on the values of random quantities, an estimated stopping criterion may incorrectly 
\begin{enumerate}
\item fail to stop the algorithm even when the deterministic gradient is sufficiently small;\footnote{The notions of ``sufficiently small'' or ``too large'' are dependent on the application, just as they are in deterministic optimization.} or
\item incorrectly stop the algorithm even though the norm of the iterate's deterministic gradient is too large. 
\end{enumerate}
An estimated stopping criterion that fails in the former respect experiences a false negative; that is, the estimated stopping criterion is not triggered even though the underlying deterministic stopping criterion would have been triggered. An estimated stopping criterion that fails in the second respect experiences a false positive; that is, the estimated criterion is triggered even though the underlying deterministic criterion would not have been triggered.
Ideally, a stopping criterion is specified in order to control the probability of both types of falsehoods from occurring.

To our knowledge, there are no stopping criteria that have been specified that have rigorously addressed either of these concerns. Therefore, in this work, we will derive false negative probability controls. While our techniques can be applied to derive coarse false positive probability controls, we will not state formal results in this direction: in our view, a false positive is often deleterious as it would terminate the algorithm too early for any utility to be derived from the terminal iterate; thus, coarse false positive controls are rather useless for most practical situations (and, for those where it is allowable, they can be readily derived using our techniques). For meaningful false positive control, we would need assumptions on the cumulative distribution function of $\Vert\dot f(\theta,X) \Vert_2$, which is an atypical consideration in the stochastic optimization literature. Thus, we will leave it to future work.

%% file: section/literature.tex
For the specification of nonconvex functions, there are many conditions in the literature and many of them are intimately related. Our goal here is to define these conditions as they appear in the literature, and demonstrate how they are related. At the end, we state the Bottou-Curtis-Nocedal (BCN) conditions, which was popularized in \cite{bottou2018} and which is the main set of conditions under consideration in this work.

Before enumerating these distinct conditions, we note that there is a core assumption that $F$ is bounded from below, both $\dot{F}$ and $\dot{f}$ exist, and that $\dot{f}$ is unbiased. While it is possible to allow for bias in $\dot{f}$ (see \S 4 of \cite{bottou2018}), this is only useful when we are using a scalar step size; if we consider a matrix-based step size, as we will do here, then the bias in $\dot{f}$ cannot be allowed without extra assumptions. Hence, we will take $\dot{f}$ to be unbiased and we simply point out that the biased case with scalar step sizes is an easier analysis than what we will present below.

\subsection{Control on the Iterate Space}

There are three common ways in which the space of iterates of $\theta$ is controlled. The first condition assumes

\condition{condition-is-compact}{The argument, $\theta$, is restricted to a compact, convex subset of $\mathbb{R}^p$, denoted by $\Theta$ \cite{lei2018}.}{IS-1}

Condition \ref{condition-is-compact} is quite convenient in the analysis of SGD methods because it prevents the iterates from diverging out of a compact set---if an iterate is pushed outside of the convex, compact set then it is simply projected back into the set. As a result of this projection operation, the analysis of SGD methods where the iterates are restricted to a convex, compact is a special case of when the situation where $\Theta$ is not required to be bounded \cite{bertsekas2011}. 

A less restrictive approach assumes
\condition{condition-is-closed}{The argument, $\theta$, is restricted to a closed, convex subset of $\mathbb{R}^p$, denoted by $\Theta$ \cite{ghadimi2016,wang2019}.}{IS-2}

Condition \ref{condition-is-closed} certainly contains Condition \ref{condition-is-compact}, and thus is more general. However, just as for Condition \ref{condition-is-compact}, Condition \ref{condition-is-closed} is imposed by projecting the iterates into $\Theta$. Again, as a result of this projection operation, the analysis of SGD methods where the iterates are restricted to a convex, closed set is a special case of the analysis in which $\Theta$ is the whole space $\mathbb{R}^p$.

Thus, the final way in which the space of iterates of $\theta$ is controlled is to allow $\theta$ to be unrestricted. For uniformity, we define this as
\condition{condition-is-unrestricted}{The argument, $\theta$, is unrestricted. That is $\theta$ can take any value in $\mathbb{R}^p$.}{IS-3}

\subsection{Control on the Objective}

There are several common ways in which the stochastic objective (SO) function, $f$, is directly controlled. The first one assumes
\condition{condition-so-bounded}{There exists a $C > 0$ such that
$$\Prb{ \forall \theta, ~ |f(\theta,X)| < C  } = 1.$$}{SO-1}

Condition \ref{condition-so-bounded} is assumed in \cite{hu2017,zhou2018}. In fact, both of these works require more stringent controls on $f$ and its derivatives. However, in \cite{hu2017}, these added controls allow for the approximation of SGD methods by stochastic differential equations, which are then leveraged to characterize the escape times of an SGD method from a saddle point. It is worth noting that such a characterization supplies a more complete analysis than the more recent results of \cite{fang2019,jin2019}, which only supply conditions for finding an $\epsilon$-approximate stationary point with a nearly positive definite Hessian, but do not guarantee that such a point is stable (i.e., that the SGD method will not escape this point as well).

A natural relaxation of Condition \ref{condition-so-bounded} is to assume that $F(\theta)$ is bounded from above and below. While bounding $F(\theta)$ from below is necessary when $\theta$ is not restricted to a compact subset, there were no works in our review that directly assumed that $F(\theta)$ is bounded from above. The closest condition in this regard come from \cite{fehrman2019}, which assumes
\condition{condition-so-2ndmoment}{There exists a $C > 0$ such that for all $\theta \in \mathbb{R}^p$,
$$\E{ |f(\theta,X)|^2 } \leq C.$$}{SO-2}

Note, by Jensen's inequality, Condition \ref{condition-so-2ndmoment} implies
\begin{equation}
|F(\theta)|^2 = | \E{ f(\theta,X) } |^2 \leq \E{ |f(\theta,X)|^2 } \leq C;
\end{equation}
that is, Condition \ref{condition-so-2ndmoment} implies $F$ is bounded. 

An alternative restriction to bounding the stochastic objective is to impose Lipschitz continuity, which is often more appropriate for a number of data-driven optimization problems. One specific form of using such continuity is assumed in \cite{reddi2016a,park2019}, and specified by 
\condition{condition-so-lipschitz}{There exists a $C > 0$ such that
$$\Prb{ \forall \theta_1,\theta_2,~ |f(\theta_1,X) - f(\theta_2,X)| \leq C \norm{\theta_1 - \theta_2}_2 } = 1.$$}{SO-3}

Unfortunately, Condition \ref{condition-so-lipschitz} excludes our litmus problem (i.e., the standard linear regression problem). Interestingly, Condition \ref{condition-so-lipschitz} is assumed locally in \cite{fehrman2019} along with twice differentiability and constraints on the expected Hessian of $F$ in order to establish local rates of convergence for nonconvex functions. 

The natural relaxations of Conditions \ref{condition-so-bounded} and \ref{condition-so-lipschitz} to the deterministic objective (DO) function, $F$, are
\condition{condition-do-bounded}{There exists a $C > 0$ such that $|F(\theta)| \leq C$ for all $\theta$.}{DO-1}
and
\condition{condition-do-lipschitz}{There exists a $C > 0$ such that for all $\theta_1$ and $\theta_2$, 
$$|F(\theta_1) - F(\theta_2)| \leq C \norm{\theta_1 - \theta_2}_2.$$}{DO-2}

By Jensen's inequality, it follows that Condition \ref{condition-so-bounded} implies Condition \ref{condition-do-bounded}, and that Condition \ref{condition-so-lipschitz} implies Condition \ref{condition-do-lipschitz}. Note, Condition \ref{condition-do-bounded} is used in \cite{hu2017} to establish an SDE approximation to the SGD iterates, and establish rather fine-resolutions results about the escape times of the SGD from saddle points. Condition \ref{condition-do-lipschitz} was used in the early nonconvex results of \cite{reddi2016a,ghadimi2016}, and was used recently by \cite{park2019} to establish that SGD finds $\epsilon$-approximate second-order stationary points with high probability. Importantly, Condition \ref{condition-do-lipschitz} is essential in Theorem 1 of \cite{liorab2018} (i.e., the aforementioned convergence with probability one result): it is used three times in the proof of the result and done in such a way that it cannot be relaxed. 

\subsection{Control on the Gradient}

There are many more conditions that are placed on the stochastic gradients (SG), $\dot{f}$, than the stochastic objective (SO), $f$, owing to the centrality of $\dot{f}$ in SGD methods. The first set of conditions will be analogous to those conditions on $f$. The most restrictive condition, integral to the results in \cite{reddi2016a,ma2017,hu2017,bi2019}, is 
\condition{condition-sg-bounded}{There exists a $C > 0$ such that
$$\Prb{ \forall \theta, ~\norm{ \dot f(\theta,X) }_2 \leq C  } = 1.$$}{SG-1}

A more applicable relative of this condition and an analogue of Condition \ref{condition-so-lipschitz}, used in \cite{reddi2016b,ma2017,huang2017,zhou2018,li2018,bassily2018,fang2019,jin2019,bi2019,park2019,wang2019}, is
\condition{condition-sg-lipschitz}{There exists a $C > 0$ such that
$$\Prb{ \forall \theta_1,\theta_2,~ \norm{ \dot f(\theta_1,X) - \dot f(\theta_2,X)}_2 \leq C \norm{ \theta_1-\theta_2}_2 } = 1.$$}{SG-2}
Condition \ref{condition-sg-lipschitz} is also used locally in the results of \cite{fehrman2019}. More importantly, Condition \ref{condition-sg-lipschitz} is integral to the previously mentioned result, Theorem 2(c) of \cite{lei2019}, which established convergence in probability. While we will not directly compare Condition \ref{condition-sg-lipschitz} to the BCN conditions, we will compare weaker conditions implied by Condition \ref{condition-sg-lipschitz} to the BCN conditions. In order to do so, we will first need to discuss common conditions placed on the noise model (NM).

The first common restriction on the noise model (NM), used in \cite{chen2018,liorab2018}, is specified by
\condition{condition-noise-bounded}{There exists a $C \geq 0$ such that
$$ \Prb{ \forall \theta,~ \norm{ \dot{f}(\theta,X) - \dot {F}(\theta)}_2 \leq C } = 1.$$}{NM-1}

The second and less restrictive noise model condition, used in \cite{ghadimi2016,li2018,fang2019,jin2019,yu2019,park2019,lei2018}, is specified by
\condition{condition-noise-var}{There exists a $C \geq 0$ such that, $\forall \theta$,
$$ \E{ \norm{ \dot{f}(\theta,X) - \dot{F}(\theta)}_2^2} \leq C.$$ }{NM-2}

As we will see subsequently, we have the following noise model condition, which is motivated by Condition \ref{condition-sg-lipschitz}.
\condition{condition-noise-obj}{There exists a $C_1,C_2 \geq 0$ such that, $\forall \theta$,
$$ \E{ \norm{ \dot{f}(\theta,X) - \dot{F}(\theta)}_2^2} \leq C_1 + C_2 F(\theta) .$$}{NM-3}

The following lemma relates Condition \ref{condition-sg-lipschitz} to Condition \ref{condition-noise-obj}. Note, in the following lemma, the parameter, $L$, is taken to be $0$ in \cite{lei2019}.
\begin{lemma} \label{sg-lipschitz-implies-bcn}
Suppose there exists $L \in \mathbb{R}$ such that $\Prb{\forall \theta, ~ f(\theta,X) \geq L} = 1$ and suppose Condition \ref{condition-sg-lipschitz} holds. Then, for all $\theta_1,\theta_2$,
\begin{equation}
\norm{ \dot{F}(\theta_1) - \dot{F}(\theta_2)}_2 \leq C \norm{\theta_1 - \theta_2}_2,
\end{equation}
and for some $C_1,C_2 \geq 0$, \ref{condition-noise-obj} holds.
\end{lemma}
\begin{proof}
The first result follows by an application of Jensen's inequality. For the second result, for any $\theta$, let
\begin{equation}
\tilde \theta = \theta - \frac{1}{C} \dot f(\theta,X).
\end{equation}
By Taylor's theorem and Condition \ref{condition-sg-lipschitz},
\begin{align}
f(\tilde{\theta},X) &\leq f(\theta,X) + \dot{f}(\theta,X)'(\tilde{\theta} - \theta) + \frac{C}{2} \norm{ \tilde{\theta} - \theta}_2^2 \\
	&\leq f(\theta,X) - \frac{1}{C} \norm{ \dot f(\theta,X) }_2^2 + \frac{1}{2C} \norm{ \dot f(\theta,X) }_2^2 \\
	&\leq f(\theta,X) - \frac{1}{C} \norm{ \dot f(\theta,X)}_2^2.
\end{align}
Therefore, since $L \leq f(\tilde \theta, X)$ with probability one,
\begin{equation}
\norm{ \dot{f}(\theta,X)}_2^2 \leq C[f(\theta,X) - f(\tilde{\theta},X)] \leq C[f(\theta,X) - L].
\end{equation}
Moreover, this inequality with Jensen's inequality implies that
\begin{equation}
\norm{ \dot{F}(\theta)}_2^2 \leq \E{ \norm{ \dot{f}(\theta,X)}_2^2} \leq C [F(\theta) - L]
\end{equation}
Therefore,
\begin{equation}
\E{ \norm{ \dot{f}(\theta,X) - \dot F(\theta)}_2^2} \leq 2 \E{ \norm{ \dot{f}(\theta,X)}_2^2} + 2 \norm{ \dot{F}(\theta)}_2^2 \leq 4 C (F(\theta)- L).
\end{equation}
If $L$ is positive, then we can choose $C_1 = 0$ and $C_2 = 4C$. If $L$ is negative, then we can choose $C_1 = -4C L$ and $C_2 = 4C$. \qed
\end{proof}

While a variant of Condition \ref{condition-noise-obj} was recently considered (see the discussion after \ref{condition-noise-grad}), a direct application of \ref{condition-noise-obj} and the Polyak-{\L}ojasiewicz (PL) condition \cite{karimi2016} provide a straightforward proof of convergence (in probability) of the objective function evaluated at the iterates to the optimal value.

The final noise model condition is assumed in the BCN conditions \cite{bottou2018}, and is specified by
\condition{condition-noise-grad}{There exists a $C_1,C_2 \geq 0$ such that, $\forall \theta \geq 0$,
$$ \E{ \norm{ \dot{f}(\theta,X) - \dot{F}(\theta)}_2^2} \leq C_1 + C_2 \norm{ \dot{F}(\theta)}_2^2.$$}{NM-4}

In general, if $\dot{F}$ is Lipschitz continuous, then \ref{condition-noise-grad} implies \ref{condition-noise-obj} by a simple analogue of Lemma \ref{sg-lipschitz-implies-bcn}. Given that $\dot{F}$ is often taken to be Lipschitz continuous in the stochastic optimization literature, \ref{condition-noise-obj} is effectively more general. Some recent works combining the two have also appeared under the names of expected smoothness \cite{khaled2020} and \cite{gower2020}, and show weak forms of convergence under this broader assumption. In fact, to strengthen their convergence results, the PL condition is employed \cite{khaled2020,gower2020}, which renders \ref{condition-noise-obj} and \ref{condition-noise-grad} equivalent. Specifically, if $\exists \mu > 0$ and an optimal objective function value $F^*$ such that
\begin{equation}
\norm{\dot{F}(\theta)}_2^2 \geq \mu [ F(\theta) - F^* ],
\end{equation}
then
\begin{equation}
\begin{aligned}
C_1 + C_2 F(\theta) &= (C_1 + C_2 F^*) + C_2 [ F(\theta) - F^*] \\
				   &\leq (C_1 + C_2 F^*) + \frac{C_2}{\mu} \norm{\dot{F}(\theta)}_2^2.
\end{aligned}
\end{equation}
We note that the PL condition, if assumed globally as in \cite{khaled2020,gower2020}, is rather a strong assumption to make as it all stationary points to have value $F^*$, which is usually taken to be the global minimum.

The final two conditions are the less restrictive implications that follow from Conditions \ref{condition-sg-bounded} and \ref{condition-sg-lipschitz} about the deterministic gradient (DG), $\dot{F}$. The first condition, found in \cite{huang2017,foster2018,chen2018,ward2018}, is
\condition{condition-dg-bounded}{There exists a $C > 0$ such that, for all $\theta$,
$$ \norm{ \dot {F}(\theta)}_2 \leq C.$$}{DG-1}
The second condition, found in \cite{liorab2018,zou2019,yu2019,lei2018}, is
\condition{condition-dg-lipschitz}{There exista a $C > 0$ such that for all $\theta_1,\theta_2$, 
$$ \norm{ \dot{F}(\theta_1) - \dot{F}(\theta_2)}_2 \leq C \norm{\theta_1 - \theta_2}_2.$$}{DG-2}

\subsection{Bottou-Curtis-Nocedal Functions} \label{subsection-bcn-specification}

The Bottou-Curtis-Nocedal (BCN) conditions, as popularized in \cite{bottou2018}, is a very general set of conditions as it---arguably---takes the least restrictive conditions from those discussed above. That is, the BCN conditions are defined as follows.

\condition{bcn}{The BCN conditions are

\begin{enumerate}
\item $F$ is bounded from below.
\item $\dot{F}$ is Lipschitz continuous (\ref{condition-dg-lipschitz}).
\item For all $\theta \in \mathbb{R}^p$, $\dot{F}(\theta) = \mathbb{E}[ \dot{f}(\theta,X)]$.
\item The stochastic gradients satisfy \ref{condition-noise-grad}.
\end{enumerate}}{BCN}

\begin{definition}
A function that satisfies the BCN conditions is referred to as a BCN function.
\end{definition}

\begin{remark} While we have not explicitly done so, we can allow for biased stochastic gradients if we use scalar step sizes. \end{remark} 

%% file: section/intro_convergence.tex
Here, we establish one of our key results: the deterministic gradient evaluated at the iterates generated by stochastic gradient descent (SGD) converge to zero with probability one for Bottou-Curtis-Nocedal (\ref{bcn}) nonconvex functions. In order to do this, we will first need to specify the precise nature of the SGD iterates that we will consider. Then, we will establish global convergence of SGD with probability one for BCN functions, which includes a broad class of convex and nonconvex functions. 

%% file: section/method.tex
\subsection{Stochastic Gradient Descent Method} \label{subsection-sgd}

Here, we will consider SGD methods with matrix-valued learning rates. In doing so, we will require the particular \ref{bcn} conditions described earlier; however, we emphasize that if the step size is scalar and the stochastic gradients are allowed to be bias, then the iterates can be analyzed with the same arguments below and with less difficulty. 

Let $\lbrace X_k : k \in \mathbb{N} \rbrace$ be independent and identically distributed random variables that have the same distribution as $X$. Let $\beta_0$ be either a fixed or random quantity in $\mathbb{R}^p$. Let $\mathcal{F}_0 = \sigma(\beta_0)$ and $\mathcal{F}_k = \sigma(\beta_0, X_1,\ldots,X_k)$ denote the corresponding elements of the usual filtration. Define the Stochastic Gradient Descent iterates $\lbrace \beta_k : k \in \mathbb{N} \rbrace$ recursively by
\begin{equation}
\beta_{k+1} = \beta_k - M_k \dot f (\beta_k, X_{k+1}),
\end{equation} 
where $\lbrace M_k \rbrace$ are matrices that satisfy:
\condition{prop-sympd}{The matrices $\lbrace M_k \rbrace$ are symmetric and positive definite.}{P1}
\condition{prop-maxeig}{There exists an $S > 0$ such that
$$\sum_{k=0}^\infty \lambda_{\max}(M_k)^2 < S,$$
where $\lambda_{\max}(\cdot)$ denotes the largest eigenvalue of the given matrix.}{P2}
\condition{prop-mineig}{The sum,
$$\sum_{k=0}^\infty \lambda_{\min}(M_k),$$
diverges, where $\lambda_{\min}(\cdot)$ is the smallest eigenvalue of the given matrix.}{P3}
\condition{prop-condnum}{The condition numbers of $\lbrace M_k \rbrace$ with respect to operator 2-norm, $\lbrace \kappa(M_k) \rbrace$, satisfy
$$ \lim_{k \to \infty} \lambda_{\max}(M_k) \kappa(M_k) = 0.$$
}{P4}
There are several remarks worth making at this point. First, if $M_k$ are scalar multiples of the identity, we would recognize properties \ref{prop-sympd} to \ref{prop-mineig} as the Robbins-Monro conditions \cite{robbins1951}, and \ref{prop-condnum} would be implied by \ref{prop-maxeig}. Second, \ref{prop-condnum} ensures that the condition number of $M_k$ cannot grow too rapidly such that subsequences of $\lbrace M_k \rbrace$ may hinder convergence (see Lemma \ref{lemma-minimality-smallest-eigenvalue}). Importantly, \ref{prop-condnum} can be relaxed so that the limit is a positive number that would depend on certain constants governing the deterministic gradient behavior and the noise model (see \eqref{eqn-relaxed-condnum-control}). 

Finally, we make a brief note about an important property of SGD that follows from the independent of $\lbrace X_k \rbrace$, which we will formalize and make use of later. Specifically, $\lbrace \beta_k \rbrace$ enjoy an analogue of the strong Markov property; that is, for any finite, stopping time $\tau$, the iterates $\lbrace \beta_{\tau+k} \rbrace$ are independent of $\mathcal{F}_{\tau}$ given $\beta_{\tau}$ and $M_\tau$. Moreover, since properties \ref{prop-sympd} to \ref{prop-condnum} still hold for $\lbrace M_{\tau+k} \rbrace$ (see Lemma \ref{lemma-pk-satisfy-mk-properties}), any property that holds for $\lbrace \beta_{\tau+k} \rbrace$ (given $\beta_{\tau}$ and $M_\tau$) must also hold for $\lbrace \beta_k \rbrace$. 

%% file: section/global_convergence.tex
\subsection{Strong Global Convergence} \label{subsection-strong-convergence}

With the formulation of SGD in hand and the nature of the BCN functions specified, we are now ready to prove the strong global convergence of the iterates; that is, we will prove that
\begin{equation}
\Prb{ \lim_{k \to \infty} \norm{ \dot{F}(\beta_k)}_2 = 0 } = 1.
\end{equation}

The proof of this result will proceed by two steps. First, we will establish that, for any $\delta > 0$,
\begin{equation} \label{eqn-revisit-io}
\Prb{ \norm{ \dot{F}(\beta_k)}_2 \leq \delta, ~i.o.} = 1,
\end{equation}
where $i.o.$ means infinitely often (see Theorem \ref{theorem-revisit}). Unfortunately, this alone will not imply convergence. Therefore, we will use this result to prove that, for any $\delta > 0$, 
\begin{equation} \label{eqn-exit-fo}
\Prb{ \norm{ \dot{F}(\beta_k)}_2 > \delta, ~i.o.} = 0,
\end{equation}
which implies that the deterministic gradient of the iterates converge to zero with probability one (see Theorem \ref{theorem-exit}).

Before detailing the results, we would like to point out the general strategy that we use, and how it is similar to, or distinct from, previous efforts.
First, we use coupling to relate the iterate sequence and a related sequence, for which we then establish an analogue of the strong Markov property mentioned previously. With these two pieces, a refinement of Zoutendijk's global convergence strategy \cite{zoutendijk1970}, and an induction argument, we are able to prove (\ref{eqn-revisit-io}). Then, we leverage (\ref{eqn-revisit-io}), the inclusion-exclusion principle, and a conditional version of the Borel-Cantelli lemma to conclude (\ref{eqn-exit-fo}). 

As mentioned, our approach adapts Zoutendijk's global convergence strategy \cite{zoutendijk1970}, which has been done previously in the stochastic optimization literature \cite{reddi2016a,reddi2016b}. However, our approach refines this argument by establishing an analogue of the strong Markov property through coupling, which we have not observed in any of the stochastic optimization literature. This allows us to state a much stronger result than what has previously been established. 

Another point of departure is that our approach avoids restating the behavior of the iterates as a martingale, which is the primary strategy when $F$ is convex (e.g., see \cite{bertsekas2011}). Our approach also avoids restating any evaluation of the iterates with respect to the objective or the gradient as a martingale, which is the strategy that is used in Theorem 2(c) of \cite{lei2019}, and which seems to require much stronger conditions than the general BCN conditions. In fact, the use of martingales only appears in the proof of the conditional Borel-Cantelli lemma, which we do not derive but rather cite from source material. 

Finally, our proof does not require explicitly bounding the behavior of the differences between the norms of the deterministic gradients of sequential iterates by the steps, $\lbrace M_k \rbrace$, to prove the deterministic gradients of the iterates converge to zero with probability one. This type of argumentation is essential to the proof of \cite{liorab2018}\footnote{This bound is required to apply Lemma 1 of \cite{liorab2018}. See the second display equations of Page 6 in \cite{liorab2018}.} and in \cite{lei2019}.\footnote{This type of bound is established in (15) and the subsequent display equation in \cite{lei2019}. The argument then essentially reestablishes Lemma 1 of \cite{liorab2018}.}

\begin{remark}
We also note that all of the inequalities and equalities below hold with probability one, even if this is not explicitly stated.
\end{remark}
\begin{remark}
We also point out that the probabilities and expectations below should be conditional on $\mathcal{F}_0$. However, to avoid the additional cumbersome notation, we will not explicitly state this.
\end{remark}

\subsubsection{Strong Markov Property and Coupling}

Our first task will be to set the stage to the analogue of strong Markov property that will be relevant in analyzing the SGD method. Let $\tau$ be a finite stopping time with respect to $\lbrace \mathcal{F}_k \rbrace$; that is, $\Prb{ \tau < \infty} = 1$. Moreover, as all of our arguments will be asymptotic in this section, we will assume also that $\Prb{ \tau \geq K} = 1$ where $K \in \mathbb{N}$ is defined in the following lemma to satisfy the following property, which will be used later. 

\begin{lemma} \label{lemma-minimality-smallest-eigenvalue}
Suppose $\lbrace M_k \rbrace$ satisfy \ref{prop-sympd} to \ref{prop-condnum}, and recall that $C$ is the Lipschitz constant in \ref{condition-dg-lipschitz} and $C_2$ is the scaling constant in \ref{condition-noise-grad}. There exists a $K \in \mathbb{N}$ such that, for all $k \geq K$ and for any eigenvalue, $\lambda$ of $M_k$,
\begin{equation} \label{eqn-grad-norm-lb}
\frac{1}{2} \lambda_{\min}(M_k) \leq \lambda - \frac{C}{2}\lambda^2 - \frac{C C_2}{2} \lambda_{\max}(M_k)^2.
\end{equation}
\end{lemma}

\begin{proof}
We begin by showing that right hand side of \eqref{eqn-grad-norm-lb} is minimized when $\lambda = \lambda_{\min}(M_k)$ for all $k$ sufficiently large and then establish \eqref{eqn-grad-norm-lb} for this choice of $\lambda$. 
We first establish
\begin{equation} \label{eqn-def-K}
\lambda_{\max}(M_k) + \lambda_{\min}(M_k) \leq \frac{2}{C}.
\end{equation}
By \ref{prop-maxeig}, $\lambda_{\max}(M_k) \to 0$. Therefore, there exists a $k_0 \in \mathbb{N}$ such that for all $k \geq k_0$, $\lambda_{\max}(M_k) \leq 1/C$. Therefore, (\ref{eqn-def-K}) follows for all $k \geq k_0$. Moreover, for every $k \geq k_0$, we see that (\ref{eqn-def-K}) holds for any other eigenvalue of $M_k$ in place of $\lambda_{\max}(M_k)$. Letting $\lambda$ denote any eigenvalue of $M_k$ and noting that $\lambda - \lambda_{\min}(M_k) \geq 0$,
\begin{equation}
(\lambda - \lambda_{\min}(M_k) ) (\lambda + \lambda_{\min}(M_k) ) \leq \frac{2}{C}(\lambda - \lambda_{\min}(M_k) ),
\end{equation}
which can be rearranged to conclude that, for any $k \geq k_0$ and for any eigenvalue $\lambda$ of $M_k$,
\begin{equation}
\begin{aligned}
\lambda_{\min}(M_k) - \frac{C}{2}\lambda_{\min}(M_k)^2 \leq \lambda - \frac{C}{2}\lambda^2.
\end{aligned}
\end{equation} 
It follows that the right hand side of \eqref{eqn-grad-norm-lb} is minimized by $\lambda_{\min}(M_k)$.

We now establish \eqref{eqn-grad-norm-lb} with $\lambda = \lambda_{\min}(M_k)$ by contradiction. Suppose that there exists a subsequence $\lbrace k_j \rbrace$ for which \eqref{eqn-grad-norm-lb} does not hold; that is, by rearranging \eqref{eqn-grad-norm-lb},
\begin{equation}
C \lambda_{\min}(M_{k_j})^2 + CC_2 \lambda_{\max}(M_{k_j})^2 > \lambda_{\min}(M_{k_j}).
\end{equation}
Taking an upper bound on the left hand side,
\begin{align}
2C(C_2 + 1) \lambda_{\max}(M_{k_j})^2 > \lambda_{\min}(M_{k_j}).
\end{align}
Hence,
\begin{equation} \label{eqn-relaxed-condnum-control}
\lambda_{\max}(M_{k_j}) \kappa( M_{k_j} ) > \frac{1}{2 C(C_2 + 1)},
\end{equation}
which contradicts \ref{prop-condnum} for $j$ sufficiently large. Hence, there is a $K \geq k_0$ such that, for all $k \geq K$, \eqref{eqn-grad-norm-lb} holds.
\qed
\end{proof}

Now, using $\tau$, we will define a sequence of iterates $\lbrace \psi_k \rbrace$ that we will eventually couple with $\lbrace \beta_k \rbrace$. To define $\lbrace \psi_k \rbrace$, let
\begin{enumerate}
\item $Z_k := X_{\tau+1+k}$ for all $k \in \mathbb{N}$.
\item $\psi_0 := \beta_{\tau + 1}$.
\item For all $k \in \lbrace 0 \rbrace \cup \mathbb{N}$, $P_k := M_{\tau+1+k}$ and
\begin{equation} \label{eqn-psi-iter}
\psi_{k+1} := \psi_k - P_k \dot f( \psi_k, Z_{k+1}) \1{ \norm{\dot F(\psi_k)}_2^2 > \delta},
\end{equation}
where $\1{ \cdot} $ is the indicator function for the given event.
\end{enumerate} 

There are two properties of these quantities that are worth noting: $\lbrace Z_k \rbrace$ are independent and identically distributed; $\lbrace P_k \rbrace$ inherit the properties of $\lbrace M_k \rbrace$ with probability one. The former is verified by Theorem 4.1.3 of \cite{durrett2010}, which states that $\lbrace Z_k \rbrace$ are mutually independent and independent of $\mathcal{F}_{\tau+1}$, and have the same distribution as $X_1$. The latter is verified by the following lemma.
\begin{lemma} \label{lemma-pk-satisfy-mk-properties}
With probability one, $\lbrace P_k \rbrace$ satisfy \ref{prop-sympd} to \ref{prop-condnum}. Moreover, with probability one, for all $k \geq 0$ and any eigenvalue $\lambda$ of $P_k$,
\begin{equation}
\frac{1}{2}\lambda_{\min}(P_k) \leq \lambda - \frac{C}{2} \lambda^2 - \frac{C C_2}{2} \lambda_{\max}(P_k)^2.
\end{equation}
\end{lemma}
\begin{proof}
The result follows from a standard divide and conquer argument. For \ref{prop-sympd},
\begin{align}
&\Prb{ P_k = P_k', P_k \succ 0 } = \Prb{ M_{\tau+1+k}' = M_{\tau+1+k}, M_{\tau+1+k} \succ 0} \\
&= \sum_{j=0}^\infty \condPrb{ M_{j + 1 +k}' = M_{j + 1 +k}, M_{j+1+k} \succ 0}{\tau = j} \Prb{ \tau = j} \\
&= \sum_{j=0}^\infty \underbrace{\Prb{ M_{j + 1 +k}' = M_{j + 1 +k}, M_{j+1+k} \succ 0 }}_{=1~\text{by \ref{prop-sympd} on } \lbrace M_k \rbrace } \Prb{ \tau=j} \\
&= \sum_{j=0}^\infty \Prb{ \tau = j} \\
&= \Prb{ \tau < \infty } = 1.
\end{align}
Similarly, for \ref{prop-maxeig},
\begin{align}
&\Prb{ \sum_{k=0}^\infty \lambda_{\max}(P_k)^2 < S} \\
&= \sum_{j=0}^\infty \underbrace{\condPrb{ \sum_{k=0}^\infty \lambda_{\max}(M_{j+1+k})^2 < S}{\tau=j}}_{=1~\text{by \ref{prop-maxeig} on } \lbrace M_k \rbrace} \Prb{ \tau=j} \\
&= \Prb{ \tau < \infty} =1.
\end{align}
The analogous arguments will show that the remaining conclusions of the lemma hold.
\qed
\end{proof}

From these properties, we see that if (\ref{eqn-psi-iter}) did not have the indicator term, then the only difference between $\lbrace \psi_k \rbrace$ and $\lbrace \beta_k \rbrace$ is the initialization---the fact that $\lbrace P_k \rbrace$ and $\lbrace M_k \rbrace$ are distinct is of little importance for our purposes as long as \ref{prop-sympd} to \ref{prop-condnum} are satisfied. Thus, we see that $\lbrace \beta_k \rbrace$ exhibit an analogue of the strong Markov property.

Now, to couple these two iterate sequences, let $\mathcal{G}_0 = \mathcal{F}_{\tau+1}$ and $\mathcal{G}_k = \sigma(\mathcal{F}_{\tau+1},Z_1,\ldots,Z_k)$, and, for $\delta >0$, define $\tau_\delta$ to be a stopping time with respect to $\lbrace \mathcal{G}_k \rbrace$ such that
\begin{equation}
\tau_\delta = \min\left\lbrace k \geq 0: \norm{ \dot{F}(\psi_k)}_2^2 \leq \delta \right\rbrace.
\end{equation} 
Then, on the event $\lbrace k \geq \tau_\delta \rbrace$, $\psi_{k+1} = \psi_k$. Moreover, on the event $\lbrace k < \tau_\delta \rbrace$,
\begin{equation}
\begin{aligned}
\psi_{k+1} &= \psi_{k} - P_k \dot{f}( \psi_k, Z_{k+1}) \\
&= \beta_{\tau+1+k} - M_{\tau+1 + k} \dot{f}( \beta_{\tau+1+k}, X_{\tau+2+k}) = \beta_{\tau+2+k}
\end{aligned}
\end{equation}
follows by induction. Therefore, for all $0 \leq k \leq \tau_\delta$, $\psi_k = \beta_{\tau+1+k}$; that is, the sequences are coupled in this interval. Owing to this coupling, we see that $\tau_\delta + 1$ is the number of iterates for $\beta_k$ to be within a ``$\delta$-region of zero gradient'' after iterate $\tau$. We now apply Zoutendijk's global convergence approach to conclude that $\tau_\delta$ is finite with probability one.

\subsubsection{Zoutendijk's Global Convergence Approach} 

We now apply Zoutendijk's global convergence approach \cite{zoutendijk1970} to $\lbrace \psi_k \rbrace$ to conclude that $\Prb{ \tau_\delta < \infty} = 1$. First, by the fundamental theorem of calculus and \ref{condition-dg-lipschitz} (recall, with constant $C$),
\begin{align}
F(\psi_{k+1}) &\leq F(\psi_k) + \dot{F}(\psi_k)'(\psi_{k+1} - \psi_k) + \frac{C}{2}\norm{ \psi_{k+1} - \psi_k}_2^2 \\
			  & \begin{aligned}
			  &=F(\psi_k) - \dot{F}(\psi_k)'P_k \dot f (\psi_k, Z_{k+1}) \1{ \norm{ \dot{F}(\psi_k)}_2^2 > \delta} \\
			  &+ \frac{C}{2} \norm{ P_k \dot f(\psi_k, Z_{k+1})}_2^2 \1{ \norm{ \dot{F}(\psi_k)}_2^2 > \delta}. 
			  \end{aligned} \label{eqn-ftc}
\end{align}
We now take the conditional expectation of the resulting inequality with respect to $\mathcal{G}_k$. Note, since $\psi_k, P_k$ are measurable with respect to $ \mathcal{G}_k$ and $Z_{k+1}$ is independent of $\mathcal{G}_k$, then $\cond{ F(\psi_k)}{\mathcal G_k} = F(\psi_k)$ and 
\begin{equation}
\cond{ \dot{F}(\psi_k)'P_k \dot f(\psi_k, Z_{k+1})}{\mathcal{G}_k} = \dot{F}(\psi_k)'P_k \dot{F}(\psi_k).
\end{equation}
For the third term in (\ref{eqn-ftc}), we will need to make use of \ref{condition-noise-grad} (with parameters $C_1,C_2 \geq 0$).
\begin{align}
&\cond{ \norm{ P_k \left[ \dot f(\psi_k, Z_{k+1}) - \dot{F}(\psi_k) + \dot{F}(\psi_k) \right]}_2^2}{\mathcal G_k} \nonumber \\
&= \cond{ \norm{P_k \left[ \dot f(\psi_k,Z_{k+1}) - \dot{F}(\psi_k) \right]}_2^2}{\mathcal{G}_k} + \dot F(\psi_k)'P_k^2 \dot F(\psi_k) \\
&\leq C_1\lambda_{\max}(P_k)^2 + C_2 \lambda_{\max}(P_k)^2 \norm{ \dot{F}(\psi_k)}_2^2 + \dot F(\psi_k)'P_k^2 \dot F(\psi_k)
\end{align}
Putting the calculation for these three terms together in (\ref{eqn-ftc}), we conclude
\begin{equation} \label{eqn-ftc-cond}
\begin{aligned}
&\cond{F(\psi_{k+1})}{\mathcal{G}_k} \leq F(\psi_k) + \frac{CC_1}{2} \lambda_{\max}(P_k)^2 -  \1{ \norm{ \dot{F}(\psi_k)}_2^2 > \delta} \\
&\times \left[\dot{F}(\psi_k)'P_k \dot F(\psi_k) - \frac{C}{2} \dot{F}(\psi_k)'P_k^2 \dot F(\psi_k) - \frac{C C_2}{2} \lambda_{\max}(P_k)^2 \norm{ \dot F(\psi_k)}_2^2 \right]. 
\end{aligned}
\end{equation}
We can find an upper bound for (\ref{eqn-ftc-cond}) by finding a lower bound for the third term in the right hand side of the inequality. In particular, we will lower bound
\begin{equation} \label{eqn-ftc-cond-min}
\min_{v \in \mathbb{S}^{p-1}} v'P_k v - \frac{C}{2} v' P_k^2 v - \frac{C C_2}{2} \lambda_{\max}(P_k)^2 \norm{v}_2^2,
\end{equation}
where $\mathbb S^{p-1}$ is the unit sphere in $\mathbb {R}^p$.\footnote{We could drop the last term in the optimization problem as it is a constant.}

Using the Schur decomposition of $P_k$,\footnote{Since $P_k$ is random, its Schur decomposition is random.} we can transform (\ref{eqn-ftc-cond-min}) into the equivalent problem
\begin{equation} \label{eqn-ftc-cond-min-2}
\min_{v \in \mathbb{S}^{p-1}} \sum_{i=1}^p \left[\lambda_i  - \frac{C}{2} \lambda_i^2 - \frac{C C_2}{2} \lambda_1^2\right] v_i^2,
\end{equation}
where $\lambda_{\max}(P_k) = \lambda_1 \geq \lambda_2 \geq \cdots \geq \lambda_p = \lambda_{\min}(P_k)$; and $v_i$ are the components of $v$. Applying Lemma \ref{lemma-pk-satisfy-mk-properties}, the solution to (\ref{eqn-ftc-cond-min-2}) is lower bounded by $\lambda_{\min}(P_k)/2$.

Plugging this lower bound into (\ref{eqn-ftc-cond}), we have that
\begin{equation}
\begin{aligned}
\cond{F(\psi_{k+1})}{\mathcal{G}_k} &\leq F(\psi_k) + \frac{CC_1}{2} \lambda_{\max}(P_k)^2 \\
&\quad - \frac{1}{2}\lambda_{\min}(P_k)\norm{ \dot{F} (\psi_k)}_2^2 \1{ \norm{ \dot{F}(\psi_k)}_2^2 > \delta}.
\end{aligned}
\end{equation}
Rearranging and applying the condition in the indicator,
\begin{equation}
\begin{aligned}
&\frac{\delta}{2}\lambda_{\min}(P_k) \1{ \norm{ \dot{F}(\psi_k)}_2^2 > \delta} \\
&\quad \leq F(\psi_k) - \cond{ F(\psi_{k+1})}{ \mathcal{G}_k} + \frac{C C_1}{2} \lambda_{\max}(P_k)^2.
\end{aligned}
\end{equation}
Now, recall that $P_k$ are measurable with respect to $\mathcal{F}_{\tau+1}$ and recall that $\mathcal{F}_{\tau+1} \subset \mathcal{G}_k$ for all $k$. Therefore,
\begin{equation}
\begin{aligned}
&\frac{\delta}{2}\lambda_{\min}(P_k) \condPrb{ \norm{ \dot{F}(\psi_k)}_2^2 > \delta }{ \mathcal{F}_{\tau+1}} \\
&\quad \leq \cond{ F(\psi_k) - F(\psi_{k+1})}{\mathcal{F}_{\tau+1}} + \frac{C C_1}{2} \lambda_{\max}(P_k)^2.
\end{aligned}
\end{equation}
Moreover, by (a) summing both sides from $k=0$ to $n \in \mathbb{N}$, (b) recalling that $F(\psi_0)$ is finite with probability one given $\mathcal{F}_{\tau+1}$, and (c) applying \ref{prop-maxeig} from Lemma \ref{lemma-pk-satisfy-mk-properties}, we conclude
\begin{equation}
\begin{aligned}
&\frac{\delta}{2} \sum_{k=0}^n \lambda_{\min}(P_k) \condPrb{ \norm{ \dot{F}(\psi_k)}_2^2 > \delta }{ \mathcal{F}_{\tau+1}} \\
&\quad \leq F(\psi_0) - \cond{ F(\psi_{n+1})}{\mathcal{F}_{\tau+1}} + \frac{SCC_1}{2}.
\end{aligned}
\end{equation}

Recall that, we have assumed that $F(\theta)$ is bounded from below by some constant $F_{l.b.}$ as a core assumption. Using this, we see that
\begin{equation}
\begin{aligned}
\frac{\delta}{2} \sum_{k=0}^n \lambda_{\min}(P_k) \condPrb{ \norm{ \dot{F}(\psi_k)}_2^2 > \delta }{ \mathcal{F}_{\tau+1}} 
\leq F(\psi_0) - F_{l.b.} + \frac{SCC_1}{2}.
\end{aligned}
\end{equation}

Moreover, note that, for all $k \geq 0$,
\begin{equation}
\condPrb{ \bigcap_{j=0}^\infty \left\lbrace \norm{ \dot{F}(\psi_j)}_2^2 > \delta \right\rbrace }{ \mathcal{F}_{\tau+1}} \leq \condPrb{ \norm{ \dot{F}(\psi_k)}_2^2 > \delta }{ \mathcal{F}_{\tau+1}}. 
\end{equation}
Therefore, for arbitrary $n$,
\begin{equation}
\condPrb{ \bigcap_{j=0}^\infty \left\lbrace \norm{ \dot{F}(\psi_j)}_2^2 > \delta \right\rbrace }{ \mathcal{F}_{\tau+1}} \leq \frac{F(\psi_0) - F_{l.b.} + S C C_1 /2 }{(\delta/2)\sum_{k=0}^n \lambda_{\min}(P_k)}
\end{equation} 
By \ref{prop-mineig} from Lemma \ref{lemma-pk-satisfy-mk-properties}, the right hand side of this inequality can be made arbitrarily small, which implies that the conditional probability on the left hand side is zero.

That is,
\begin{equation}
0 = \condPrb{ \bigcap_{j=0}^\infty  \left\lbrace \norm{ \dot{F}(\psi_j)}_2^2 > \delta \right\rbrace }{ \mathcal{F}_{\tau+1}} = \condPrb{ \tau_{\delta} = \infty}{\mathcal{F}_{\tau+1}}.
\end{equation}
In other words, we have concluded that $\condPrb{ \tau_\delta < \infty}{\mathcal{F}_{\tau+1}} = 1$ with probability one for any finite stopping time $\tau$. With this result, our last step is to use induction. 

For now, we will say that the iterates are in a $\delta$-region of zero gradient if the squared-norm of the gradient of the iterate is no greater than $\delta$. If we let $\tau = -1$, then $\tau_\delta$ is the first time the iterates enter a $\delta$-region of zero gradient. Let $T_1(\delta) = \tau_\delta$ when $\tau = -1$. Then, from the above argument, we have shown that $T_1(\delta)$ is a finite stopping time. Now, define $T_j(\delta)$ to be the $j^{th}$ time that the iterates enter a $\delta$-region of zero gradient. Suppose that $T_j(\delta)$ is finite. Then, define $\tau = T_j(\delta)$. Then, $\tau_\delta$ for this $\tau$ is the next time that the iterates enter a $\delta$-region of zero gradient. That is, $T_{j+1}(\delta) = \tau_\delta + T_j(\delta)$. Since we have assumed that $\tau = T_{j}(\delta)$ is finite, we conclude that $\tau_\delta$ is finite, which implies that $T_{j+1}(\delta)$ is finite. Therefore, by induction we have proven the following result.

\begin{theorem} \label{theorem-revisit}
Let $F$ be a Bottou-Curtis-Nocedal function (\ref{bcn}) and let $\lbrace \beta_k \rbrace$ be the iterates generated by Stochastic Gradient Descent satisfying \ref{prop-sympd} to \ref{prop-condnum} (\S \ref{subsection-sgd}), then, for any $\delta > 0$,
\begin{equation}
\condPrb{ \norm{ \dot F( \beta_k)}_2^2 \leq \delta, ~i.o.}{\mathcal{F}_0} = 1 ~\text{with probability 1},
\end{equation}
where $\mathcal{F}_0 = \sigma(\beta_0)$.
\end{theorem}

\subsubsection{Inclusion-Exclusion and Markov's Inequality}

Our next step is to prove that
\begin{equation}
\condPrb{ \norm{ \dot F(\beta_k)}_2 > \delta, ~i.o.}{\mathcal{F}_0} = 0 ~ \text{with probability 1}.
\end{equation}

Again, we will temporarily drop the conditioning on $\mathcal{F}_0$ for simplicity of the notation. By Theorem \ref{theorem-revisit} and the inclusion-exclusion principle, 
\begin{align}
1 &= \Prb{\left\lbrace \norm{ \dot F(\beta_k)}_2 \leq \delta, ~i.o. \right\rbrace \cup \left\lbrace \norm{ \dot F(\beta_k)}_2 > \delta, ~i.o. \right\rbrace } \\
  &\begin{aligned}
  &= \Prb{ \norm{ \dot F(\beta_k)}_2 \leq \delta, ~i.o.} + \Prb{\norm{ \dot F(\beta_k)}_2 > \delta, ~i.o.} \\
  &- \Prb{ \left\lbrace \norm{ \dot F(\beta_k)}_2 \leq \delta, ~i.o. \right\rbrace \cap \left\lbrace \norm{ \dot F(\beta_k)}_2 > \delta, ~i.o. \right\rbrace }.
  \end{aligned}
\end{align}
Applying Theorem \ref{theorem-revisit} again, we conclude that
\begin{equation} \label{eqn-inc-exc}
\begin{aligned}
&\Prb{\norm{ \dot F(\beta_k)}_2 > \delta, ~i.o.} \\
&\quad = \Prb{ \left\lbrace \norm{ \dot F(\beta_k)}_2 \leq \delta, ~i.o. \right\rbrace \cap \left\lbrace \norm{ \dot F(\beta_k)}_2 > \delta, ~i.o. \right\rbrace }.
\end{aligned}
\end{equation}

We will now show that the probability of the right hand side is zero. Note, for any outcome
\begin{equation}
\omega \in \left\lbrace \norm{ \dot F(\beta_k)}_2 \leq \delta, ~i.o. \right\rbrace \cap \left\lbrace \norm{ \dot F(\beta_k)}_2 > \delta, ~i.o. \right\rbrace,
\end{equation}
there must be an infinite subsequence of $\mathbb{N}$ such that $\beta_k$ is in a $\delta^2$-region of zero gradient and then $\beta_{k+1}$ exits this $\delta^2$-region of zero gradient. Suppose this were not true. Then, there are two cases. In the first case, $\beta_k$ enters a $\delta^2$-region of zero gradient and then never leaves, in which case
\begin{equation}
\omega \not\in \left\lbrace \norm{ \dot F(\beta_k)}_2 > \delta, ~i.o. \right\rbrace.
\end{equation}
In the second case, we have that $\beta_k$ exits a $\delta^2$-region of zero gradient, and never enters again, which implies 
\begin{equation}
\omega \not\in \left\lbrace \norm{ \dot F(\beta_k)}_2 \leq \delta, ~i.o. \right\rbrace.
\end{equation}
In both cases, we have a contradiction. Therefore, using just one of the cases, we conclude that
\begin{equation} \label{eqn-set-inc-1}
\begin{aligned}
&\left\lbrace \norm{ \dot F(\beta_k)}_2 \leq \delta, ~i.o. \right\rbrace \cap \left\lbrace \norm{ \dot F(\beta_k)}_2 > \delta, ~i.o. \right\rbrace \\
&\quad \subset \left\lbrace \norm{\dot{F}(\beta_k)}_2 \leq \delta,~ \norm{\dot{F}(\beta_{k+1})}_2 > \delta, ~i.o. \right\rbrace.
\end{aligned}
\end{equation}

We can write this latter event as
\begin{equation} \label{eqn-set-inc-2}
\begin{aligned}
&\left\lbrace \norm{\dot{F}(\beta_k)}_2 \leq \delta,~ \norm{\dot{F}(\beta_{k+1})}_2 > \delta, ~i.o. \right\rbrace \\
&\quad = 
\left\lbrace \norm{ \dot{F}(\beta_{k+1})}_2 \1{ \norm{\dot{F}(\beta_k)}_2 \leq \delta} > \delta, ~i.o. \right\rbrace.
\end{aligned}
\end{equation}

We will now show that this ultimate event occurs with probability zero using Markov's inequality and the Borel-Cantelli lemma. Let $\epsilon > 0$ and recall that $C > 0$ is the parameter in \ref{condition-dg-lipschitz}. Then,
\begin{align}
&\condPrb{ \norm{ \dot{F}(\beta_{k+1})}_2 \1{ \norm{\dot{F}(\beta_k)}_2 \leq \delta} \geq \delta + C \epsilon }{\mathcal{F}_k} \\
&\begin{aligned}
&\leq \mathbb{P}\bigg{[} \left( \norm{ \dot{F}(\beta_{k+1}) - \dot{F}(\beta_k) }_2 + \norm{ \dot{F}(\beta_k)}_2 \right) \1{ \norm{\dot{F}(\beta_k)}_2\leq \delta}  \\
&\quad  \geq \delta + C \epsilon \bigg{\vert} \mathcal{F}_k \bigg{]}
\end{aligned} \\
&\leq \condPrb{ \norm{ \dot{F}(\beta_{k+1}) - \dot{F}(\beta_k)}_2 \1{ \norm{\dot{F}(\beta_k)}_2 \leq \delta} + \delta \geq \delta + C \epsilon}{\mathcal{F}_k}\\
& \leq \condPrb{ C \norm{\beta_{k+1} - \beta_k}_2 \1{ \norm{\dot{F}(\beta_k)}_2 \leq \delta} \geq  C \epsilon}{\mathcal{F}_k} \\
& \leq \condPrb{ \norm{M_k \dot{f}(\beta_k, X_{k+1})}_2 \1{ \norm{\dot{F}(\beta_k)}_2 \leq \delta} \geq   \epsilon}{\mathcal{F}_k}
\end{align}

Applying Markov's inequality to the last conditional probability and using \ref{condition-noise-grad},
\begin{align}
&\condPrb{ \norm{ \dot{F}(\beta_{k+1})}_2 \1{ \norm{\dot{F}(\beta_k)}_2 \leq \delta} \geq \delta + C \epsilon }{\mathcal{F}_k} \\
& \leq \frac{\lambda_{\max}(M_k)^2}{\epsilon^2}\left[ C_1 + (C_2+1) \norm{ \dot{F}(\beta_k)}_2^2 \right] \1{ \norm{ \dot{F}(\beta_k)}_2 \leq \delta}  \\
& \leq \frac{\lambda_{\max}(M_k)^2}{\epsilon^2}\left[ C_1 + (C_2+1)\delta^2 \right]
\end{align}
By \ref{prop-maxeig}, the sum of the right hand side is bounded with probability one. Therefore, by the conditional second Borel-Cantelli lemma (Theorem 5.3.2 of \cite{durrett2010}),
\begin{equation}
\Prb{\norm{ \dot{F}(\beta_{k+1})}_2 \1{ \norm{\dot{F}(\beta_k)}_2 \leq \delta} \geq \delta + C \epsilon, ~i.o.} = 0. 
\end{equation}
Since $\epsilon > 0$ is arbitrary, then this conclusion holds for each element in the sequence $\lbrace \epsilon_m \rbrace$ where $\epsilon_m \downarrow 0$. Therefore,
\begin{align}
&\Prb{\norm{ \dot{F}(\beta_{k+1})}_2 \1{ \norm{\dot{F}(\beta_k)}_2 \leq \delta} > \delta, ~i.o.} \\
&=\Prb{\bigcup_{m \in \mathbb{N}} \left\lbrace \norm{ \dot{F}(\beta_{k+1})}_2 \1{ \norm{\dot{F}(\beta_k)}_2 \leq \delta} \geq \delta + C \epsilon_m, ~i.o. \right\rbrace} \\
&\leq \sum_{m=1}^\infty \Prb{\norm{ \dot{F}(\beta_{k+1})}_2 \1{ \norm{\dot{F}(\beta_k)}_2 \leq \delta} \geq \delta + C \epsilon_m, ~i.o.} \\
&= 0.  
\end{align}

Therefore, by using this result with (\ref{eqn-inc-exc}), (\ref{eqn-set-inc-1}) and (\ref{eqn-set-inc-2}), we conclude the following result.
\begin{theorem} \label{theorem-exit}
Let $F$ be a Bottou-Curtis-Nocedal function (\ref{bcn}) and let $\lbrace \beta_k \rbrace$ be the iterates generated by Stochastic Gradient Descent satisfying \ref{prop-sympd} to \ref{prop-condnum} (\S \ref{subsection-sgd}), then, for any $\delta > 0$,
\begin{equation}
\condPrb{ \norm{ \dot F( \beta_k)}_2 > \delta, ~i.o.}{\mathcal{F}_0} = 0 ~\text{with probability 1},
\end{equation}
where $\mathcal{F}_0 = \sigma(\beta_0)$.
\end{theorem}

Theorem \ref{theorem-exit} supplies the following corollary. 
\begin{corollary} \label{corollary-convergence}
Let $F$ be a Bottou-Curtis-Nocedal function (\ref{bcn}) and let $\lbrace \beta_k \rbrace$ be the iterates generated by Stochastic Gradient Descent satisfying \ref{prop-sympd} to \ref{prop-condnum} (\S \ref{subsection-sgd}), then
\begin{equation}
\condPrb{ \lim_{k \to \infty} \norm{ \dot{F} (\beta_k)}_2 = 0}{\mathcal{F}_0} = 1, ~\text{with probability 1},
\end{equation}
where $\mathcal{F}_0 = \sigma(\beta_0)$.
\end{corollary}
\begin{proof}
For any $\delta > 0$, by Theorem \ref{theorem-exit},
\begin{align}
1 &= \condPrb{ \left\lbrace \norm{ \dot F( \beta_k)}_2 > \delta, ~i.o. \right\rbrace^c}{\mathcal{F}_0} \\
  &= \condPrb{ \limsup_{k \to \infty} \norm{ \dot{F}(\beta_k)}_2 \leq \delta}{\mathcal{F}_0}.
\end{align}
Since $\delta > 0$ is arbitrary, the preceding result applies to each element in the sequence $\lbrace \delta_m \rbrace$ where $\delta_m \downarrow 0$. Since the countable intersection of probability one events has probability one,
\begin{align}
1 &= \condPrb{ \bigcap_{m \in \mathbb{N}} \left\lbrace \limsup_{k \to \infty} \norm{ \dot{F}(\beta_k)}_2 \leq \delta_m \right\rbrace }{\mathcal{F}_0} \\
  &= \condPrb{ \limsup_{k \to \infty} \norm{ \dot{F}(\beta_k)}_2 = 0}{\mathcal{F}_0},
\end{align}
which is the desired result. \qed
\end{proof}

%% file: section/intro_stopping_criteria.tex
As discussed in \S\ref{section-rigorous}, the key challenges related to rigorous estimated stopping criteria are:
establishing that the underlying deterministic stopping criterion will be triggered;
establishing that the estimated stopping criterion will be triggered;
establishing false positive control;\footnote{Recall that a false positive occurs when the estimated stopping criterion is triggered, but the underlying deterministic stopping criterion, if it could be evaluated, would not have been triggered.} and
establishing false negative control.\footnote{Recall that a false negative occurs when the estimated stopping criterion is not triggered, but the underlying deterministic stopping criterion, if it could be evaluated, would have been triggered.}

We have addressed the first challenge in Corollary \ref{corollary-convergence} when the underlying deterministic stopping criterion relies on the gradient function becoming sufficiently small in norm. Therefore, we now address whether the estimated stopping criterion will be triggered, and the issue of false negative control. We again underscore that while the probability of false positives can be controlled using the techniques discussed below, such bounds would be rather coarse and would not be useful in the typical optimization context where a false positive would be particularly insidious. 

To achieve this, we organize the remainder of this section as follows. In \S\ref{subsection-sc-overview}, we state our two estimated stopping criteria (\ref{stop-mean-est} and \ref{stop-vote-ind}) that correspond to the aforementioned gradient-based deterministic stopping criterion. In \S\ref{subsection-sc-scenarios}, we superficially examine the consequences of the \ref{bcn} conditions for our stopping criteria, which motivate two important specializations that we will consider in addition to the general \ref{bcn} conditions. In \S\ref{subsection-sc-grad-est}, we establish that \ref{stop-mean-est} will be triggered with probability one and establish control over its false negative probability. In \S\ref{subsection-sc-maj-vote}, we establish that \ref{stop-vote-ind} will be triggered with probability one and establish control over its false negative probability. 

%% file: section/statement_stopping_criteria.tex
\subsection{Overview of Estimated Stopping Criteria} \label{subsection-sc-overview}

Now, we specify the estimated stopping criteria for the underlying deterministic criterion. First, we recall that the underlying deterministic stopping criterion is when the norm of the gradient drops below a given value. Second, before stating these two criteria, we note that a stopping criterion \textit{should not} be evaluated at each iteration in the stochastic setting, but rather ought to be evaluated at select, possibly random, iterations. While we will not specify the selection of these iterations in this work, we will allow for this generality by considering the situation in which the given estimated stopping criterion is evaluated at a strictly increasing sequence of finite stopping times, $\lbrace T_j : j \in \mathbb{N} \rbrace$, with respect to $\lbrace \mathcal{F}_k \rbrace$.

The first estimated stopping criterion directly attempts to estimate the gradient function using independent samples, while the second estimated stopping criterion attempts to estimate the zero-one question posed by the deterministic stopping criterion using independent samples. 

\condition{stop-mean-est}{Let $\epsilon > 0$. Let $\lbrace N_j \rbrace$ be $\mathbb{N}$-valued random variables such that $N_j$ is measurable with respect to $\mathcal{F}_{T_j}$. Moreover, for each $j$, let $\lbrace Z_{ij}: i = 1,\ldots,N_j \rbrace$ be copies of $X$ that are independent of each other and $\lbrace \mathcal{F}_k \rbrace$. Then, the SGD iterates are stopped at iterate $T_J$, where
$$J = \min\left\lbrace j \geq 1 : \frac{1}{N_j} \norm{ \sum_{i=1}^{N_j} \dot f(\beta_{T_j}, Z_{ij})}_2 \leq \epsilon \right\rbrace.$$
}{SC-1}

\condition{stop-vote-ind}{Let $\epsilon > 0$. Let $\lbrace N_j \rbrace$ be $\mathbb{N}$-valued random variables such that $N_j$ is measurable with respect to $\mathcal{F}_{T_j}$. Let $\bar \delta \in (0,1)$ and let $\lbrace \delta_j \rbrace$ be $(0,\bar \delta)$-valued random variables such that $\delta_j$ is measurable with respect to $\mathcal{F}_{T_j}$. Moreover, for each $j$, let $\lbrace Z_{ij}: i = 1,\ldots,N_j \rbrace$ be copies of $X$ that are independent of each other and $\lbrace \mathcal{F}_k \rbrace$. Then, the SGD iterates are stopped at iterate $T_J$, where
$$
J = \min\left\lbrace j \geq 1 : \frac{1}{N_j} \sum_{i=1}^{N_j} \1{ \norm{ \dot f (\beta_{T_j}, Z_{ij})}_2 \leq \epsilon} \geq \delta_j \right\rbrace.
$$}{SC-2}

\begin{remark}
We can readily imagine more efficient, dependent variants of these two stopping criteria that make use of a recent set of historical values of the stochastic gradients used to update $\lbrace \beta_k \rbrace$. Indeed, such stopping criteria would be the most desirable and require an even more complex analysis, which we will leave to future work.
\end{remark}

We now address one final technical point. Given that our estimated stopping criteria are only evaluated at iterations $\lbrace T_j : j \in \mathbb{N} \rbrace$, we need to ensure that the underlying deterministic stopping criterion would be triggered at these iterations in finite time with probability one. Thus, our first step is to specialize Corollary \ref{corollary-convergence} to these iterations as follows.

\begin{corollary}\label{corollary-ST-converge}
Let $F$ be a Bottou-Curtis-Nocedal nonconvex function (\S \ref{subsection-bcn-specification}) and let $\lbrace \beta_k \rbrace$ be the iterates generated by Stochastic Gradient Descent satisfying \ref{prop-sympd} to \ref{prop-condnum} (\S \ref{subsection-sgd}). If $\lbrace T_j \rbrace$ are positive-valued, strictly increasing, finite stopping times with respect to $\lbrace \mathcal{F}_k \rbrace$, then
\begin{equation}
\condPrb{ \lim_{j \to \infty} \norm{ \dot{F}(\beta_{T_j}) }_2 = 0 }{\mathcal{F}_0} = 1, ~\text{with probability 1},
\end{equation}
where $\mathcal{F}_0 = \sigma(\beta_0)$. Therefore, for any $\epsilon > 0$, there exist a stopping time $J^*$ that is finite with probability one such that 
$T_{J^*}$ is finite with probability one, and
\begin{equation}
\condPrb{ \norm{ \dot{F}(\beta_{T_{J^*}})}_2 \leq \epsilon }{\mathcal{F}_0} = 1, ~ \text{with probability 1}.
\end{equation}
\end{corollary}
\begin{proof}
Note, $\lbrace T_j \rbrace$ are strictly increasing, which implies $T_j \geq j - 1$ for all $j$. Therefore, for any $j \in \mathbb{N}$ and $\epsilon > 0$,
\begin{align}
&\bigcap_{m = j}^\infty \left\lbrace \norm{ \dot F (\beta_{T_m} ) }_2 \leq \epsilon \right\rbrace  \\
&\quad= \bigcap_{m = j}^\infty \left( \bigcup_{k=m-1}^\infty \left\lbrace \norm{ \dot F (\beta_k ) }_2 \leq \epsilon \right\rbrace \cap \left\lbrace T_m = k \right\rbrace  \right) \\
&\quad\supset \bigcap_{m = j}^\infty \left( \bigcup_{k=m-1}^\infty \left\lbrace \sup_{k \geq j -1} \norm{ \dot{F}(\beta_k) }_2 \leq \epsilon \right\rbrace \cap \left\lbrace T_m = k \right\rbrace \right) \\
&\quad = \bigcap_{m=j}^\infty \left[ \left\lbrace \sup_{k \geq j -1} \norm{ \dot{F}(\beta_k) }_2 \leq \epsilon \right\rbrace \cap \left( \bigcup_{k=m-1}^\infty \lbrace T_m = k \rbrace  \right) \right] \\
&\quad = \bigcap_{m=j}^\infty  \left\lbrace \sup_{k \geq j -1} \norm{ \dot{F}(\beta_k) }_2 \leq \epsilon \right\rbrace \cap \lbrace T_m < \infty \rbrace  \\
&\quad = \left\lbrace \sup_{k \geq j -1} \norm{ \dot{F}(\beta_k) }_2 \leq \epsilon \right\rbrace \cap \left(  \bigcap_{m=j}^\infty  \lbrace T_m < \infty \rbrace \right).
\end{align}
Since $T_m$ are assumed to be finite stopping times with probability one, we conclude
\begin{equation}
\condPrb{ \sup_{T_m \geq j-1} \norm{ \dot{F}(\beta_{T_m}) }_2 \leq \epsilon }{\mathcal{F}_0} \geq \condPrb{ \sup_{k \geq j-1} \norm{ \dot F (\beta_k) }_2 \leq \epsilon}{\mathcal{F}_0}.
\end{equation}
By Corollary \ref{corollary-convergence}, the limit of the left hand side converges to $1$ as $j \to \infty$. Since $\epsilon > 0$ is arbitrary, the first part of the result follows.

For the second part, recall that $T_j$ are finite with probability one for all $j$. Therefore,
\begin{equation}
\begin{aligned}
\condPrb{ T_{J^*} < \infty}{\mathcal{F}_0} &= \sum_{j=0}^\infty \condPrb{ T_j < \infty}{J^* = j, \mathcal{F}_0} \condPrb{J^* = j}{\mathcal{F}_0} \\
& = \sum_{j=0}^\infty \condPrb{J^*=j}{\mathcal{F}_0} = \condPrb{ J^* < \infty}{\mathcal{F}_0}.
\end{aligned}
\end{equation}
Hence, it is enough to prove that $\condPrb{ J^* < \infty}{\mathcal{F}_0} = 1$. Note,
\begin{align}
\condPrb{ J^* < \infty }{\mathcal{F}_0} &= \condPrb{ \bigcup_{j} \left\lbrace \norm{ \dot{F}(\beta_{T_j})}_2 \leq \epsilon \right\rbrace }{\mathcal{F}_0} \\
&\geq \condPrb{ \limsup_{j \to \infty} \norm{ \dot{F}(\beta_{T_j})}_2 \leq \epsilon}{\mathcal{F}_0},
\end{align}
where we have already shown that the right hand side has probability one. \qed
\end{proof}

%% file: section/statement_scenarios.tex
\subsection{Consequences of the BCN Conditions and Specializations} \label{subsection-sc-scenarios}

For a \ref{bcn} function, we will readily be able to establish that \ref{stop-mean-est} will be triggered with probability one. We will refer to the case where we consider a function satisfying \ref{bcn} as Scenario (a). Unfortunately, we will have some difficulty with this general scenario for \ref{stop-vote-ind} if the stochastic gradients are concentrated away from their mean, which happens for the \ref{bcn} function induced by $f(\theta,X) = \theta X$, where $X$ is a Rademacher random variable. 

For this reason, we will consider two specializations of \ref{bcn}. The first specialization is the case when $C_1 = 0$ in the definition of \ref{condition-noise-grad}. This specialization corresponds to the important case of over-parametrized models, and occurs when the minimizer of $F$ is also a minimizer of all $f$ with probability one \cite{bassily2018}. This first specialization will be referred to as Scenario (b).

For the final specialization of \ref{bcn}, which we refer to as Scenario (c), we consider the following additional condition.
\condition{condition-tail-pareto}{Let $\pi_1, \pi_2 \in (0,1)$ and $\pi_3 \geq 1$. For $\norm{\dot{F}(\theta)}_2 \leq \pi_1$ and for any $t \geq \pi_3 \norm{\dot{F}(\theta)}_2$,
$$\Prb{ \norm{ \dot{f}(\theta,X)}_2 \geq t} \leq 
\left(\frac{\pi_3 \norm{ \dot{F}(\theta)}_2}{t}\right)^{\pi_2} 
$$}{T1}
The role of \ref{condition-tail-pareto} is to ensure that there is a concentration of the norm of the stochastic gradient function near small values when the deterministic gradient is sufficiently small, which would avoid the issues that occur with $f(\theta,X) = \theta X$, where $X$ is a Rademacher random variable. Moreover, we note that the general \ref{bcn} conditions do not imply \ref{condition-tail-pareto}, as illustrated by $f(\theta,X) = \theta X$. Conversely, \ref{condition-tail-pareto} does not imply \ref{bcn}, as a random variable that satisfies \ref{condition-tail-pareto} is not even guaranteed to have a mean, leave alone a second moment (i.e., \ref{condition-noise-grad}).

To summarize, Scenario (a) refers to objective functions satisfying the general \ref{bcn} conditions; Scenario (b) refers to objective functions satisfying the \ref{bcn} conditions with $C_1 = 0$ in \ref{condition-noise-grad}; and Scenario (c) refers to objective functions satisfying \ref{bcn} and \ref{condition-tail-pareto}. 

%% file: section/sc_direct_estimate.tex
\subsection{Stopping Criterion by Gradient Estimation} \label{subsection-sc-grad-est}
We now establish that \ref{stop-mean-est} will be triggered with probability one and we derive bounds on the probability of a false negative. We begin with some general consequences of any effort to estimate the mean for the three aforementioned scenarios.

\begin{lemma} \label{lemma-prob-bound}
Let $F$ be a \ref{bcn} function (\S \ref{subsection-bcn-specification}). For $N \in \mathbb{N}$, let $\lbrace Z_1,\ldots,Z_N \rbrace$ be independent copies of $X$. Then for $\epsilon > 0$,
\begin{equation}
\begin{aligned}
&\Prb{ \frac{1}{N} \norm{ \sum_{i=1}^N \dot f(\theta, Z_i) }_2 \leq \epsilon } 
\geq 1 - \frac{C_1 + (C_2+N) \norm{\dot{F}(\theta)}_2^2}{N \epsilon^2} .
\end{aligned}
\end{equation}
Moreover, if \ref{condition-tail-pareto} holds and $\norm{ \dot{F}(\theta)}_2 \leq \pi_1$, then, for $\epsilon \geq \pi_3 \norm{ \dot{F}(\theta)}_2$,
\begin{equation}
\Prb{ \frac{1}{N} \norm{ \sum_{i=1}^N \dot f(\theta, Z_i) }_2 \leq \epsilon} \geq 
1 - N \left( \frac{\pi_3 \norm{ \dot{F}(\theta)}_2}{\epsilon} \right)^{\pi_2}.
\end{equation}
\end{lemma}
\begin{proof}
In both cases, we will find the upper bound for the complement, from which a lower bound for the stated event is readily derived. For the first case, by Markov's inequality and \ref{condition-noise-grad},
\begin{align}
\Prb{ \frac{1}{N} \norm{ \sum_{i=1}^N \dot f(\theta, Z_i) }_2 > \epsilon } & \leq \frac{\E{ \norm{ \frac{1}{N} \sum_{i=1}^{N} \dot f(\theta,Z_i)  }_2^2 } }{\epsilon^2} \\
&\leq \frac{C_1 + (C_2+N)\norm{\dot F(\theta)}_2^2 }{N\epsilon^2}.
\end{align} 
For the second case, if $\epsilon \geq \pi_3 \norm{ \dot{F}(\theta)}_2$, then
\begin{align}
\Prb{ \frac{1}{N} \norm{ \sum_{i=1}^N \dot f(\theta, Z_i) }_2 > \epsilon } & \leq N \Prb{ \norm{ \dot f(\theta,Z_1) }_2 > \epsilon},
\end{align}
to which we apply \ref{condition-tail-pareto}. \qed
\end{proof}

The first part of Lemma \ref{lemma-prob-bound} covers Scenarios (a) and (b), while the second part covers Scenario (c). We discuss each of these scenarios individually and in order. 
With regard to Scenario (a), Lemma \ref{lemma-prob-bound} shows that the main hurdle in establishing a meaningful bound on the probability of a false negative is in choosing $N \epsilon^2$ to be sufficiently large to account for the constant $C_1 > 0$. Thus, for applications where achieving a small bound on the deterministic gradient is necessary, Lemma \ref{lemma-prob-bound} implies that $N$ would need to be larger than the reciprocal squared of this small bound. On the other hand, for applications where achieving a bound on the order of the noise level is acceptable (e.g., $\epsilon \approx \sqrt{C_1}$), then moderate values of $N$ are acceptable. The following result formalizes these concepts.

\begin{proposition} \label{prop-SC-2-detect-fn}
Let $F$ be a \ref{bcn} function (\S \ref{subsection-bcn-specification}) and let $\lbrace \beta_k \rbrace$ be the iterates generated by Stochastic Gradient Descent satisfying \ref{prop-sympd} to \ref{prop-condnum} (\S \ref{subsection-sgd}). Let $\epsilon > 0$. If 
\begin{equation} \label{eqn-ind-scenario-a-sample}
\condPrb{\liminf_{j \to \infty} N_j > \frac{C_1}{\epsilon^2}}{\mathcal{F}_0} = 1 ~\text{with probability one},
\end{equation}
then \ref{stop-mean-est} will be triggered in finite time with probability one. Moreover, for any $\rho \in (0,1)$ and $\gamma > 1$, if
\begin{equation}
N_j > \frac{\gamma C_1 + C_2 ( \epsilon \rho )^2}{(1 - \rho^2) \epsilon^2},
\end{equation}
then
\begin{equation}
\condPrb{ \frac{1}{N_j}\norm{ \sum_{i=1}^{N_j} \dot f(\beta_{T_j},Z_{ij}) }_2 > \epsilon ,~\norm{ \dot{F}(\beta_{T_j})}_2 \leq \frac{\rho \epsilon}{\sqrt{\gamma}} }{\mathcal{F}_{T_j}} \leq \frac{1}{\gamma}.
\end{equation}
\end{proposition}
\begin{proof}
When $C_1 = 0$, the proof of the result uses analogous reasoning to the case when $C_1 > 0$; therefore, we will take $C_1 > 0$. 
By \eqref{eqn-ind-scenario-a-sample}, there exist $\delta > 1$ and a finite stopping time, $\bar{J}$, such that for all $j \geq \bar{J}$, with probability one,
\begin{equation}
N_j \geq \frac{\delta C_1}{\epsilon^2}.
\end{equation}
Moreover, by Corollary \ref{corollary-ST-converge}, for $\delta' \in (1, \delta)$, there exists a finite stopping time $J^*$ such that for all $j \geq J^*$, with probability one,
\begin{equation}
\frac{(\delta' - 1) C_1}{C_2 + \delta C_1 / \epsilon^2 } > \norm{ \dot F(\beta_{T_j}) }_2^2.
\end{equation}
Therefore, for all $j \geq \max \lbrace J^*, \bar{J} \rbrace$, with probability one, 
\begin{align}
\frac{C_1 + (C_2 + N_j) \norm{ \dot F (\beta_{T_j} ) }_2^2}{N_j \epsilon^2} & 
< \frac{C_1 + (C_2 + \delta C_1/\epsilon^2)\norm{ \dot F (\beta_{T_j} ) }_2^2}{\delta C_1} \\
&<\frac{C_1 + (\delta' - 1) C_1}{\delta C_1}.
\end{align}
Therefore, for all $j \geq \max \lbrace J^*, \bar{J} \rbrace$, Lemma \ref{lemma-prob-bound} implies
\begin{equation}
\condPrb{\frac{1}{N_j}\norm{ \sum_{i=1}^{N_j} \dot f(\beta_{T_j},Z_{ij}) }_2 \leq \epsilon }{\mathcal{F}_{T_j}} 
\geq 1 - \frac{\delta'}{\delta} > 0.
\end{equation}
Given that $\lbrace Z_{ij} \rbrace$ are independent over $i$ and $j$ and recall that $J$ is defined in \ref{stop-mean-est}, the probability that $J > j$ is controlled by a geometric distribution for $j \geq \max \lbrace J^*, \bar{J} \rbrace$. Therefore, we conclude that $J$ is finite with probability one: that is, \ref{stop-mean-est} is triggered in finite time with probability one. 

For the second part of the result, note
\begin{equation}
\begin{aligned}
& \condPrb{ \frac{1}{N_j}\norm{ \sum_{i=1}^{N_j} \dot f(\beta_{T_j},Z_{ij}) }_2 > \epsilon ,~\norm{ \dot{F}(\beta_{T_j})}_2 \leq \frac{\rho \epsilon}{\sqrt{\gamma}} }{\mathcal{F}_{T_j}} \\
&= \condPrb{ \frac{1}{N_j}\norm{ \sum_{i=1}^{N_j} \dot f(\beta_{T_j},Z_{ij}) }_2 > \epsilon }{\mathcal{F}_{T_j}} \1{\norm{ \dot{F}(\beta_{T_j})}_2 \leq \frac{\rho \epsilon}{\sqrt{\gamma}} }. 
\end{aligned}
\end{equation}
Applying Lemma \ref{lemma-prob-bound} with $\Vert \dot F ( \beta_{T_j} ) \Vert_2 \leq \rho \epsilon / \sqrt{\gamma}$ and the lower bound on $N_j$ implies the bound on the false negative probability.
\qed
\end{proof}

Clearly, Proposition \ref{prop-SC-2-detect-fn} applies to Scenario (b) which, we recall, is the special case of Scenario (a) in which $C_1 = 0$. Here, we see that even with $N_j = 1$ for all $j$, Proposition \ref{prop-SC-2-detect-fn} implies that \ref{stop-mean-est} would be triggered in finite time (though the false positive control would likely be quite poor). Similarly, Proposition \ref{prop-SC-2-detect-fn} allows for controlling the false negative probability at $1/\gamma$ so long as
\begin{equation}
N_j > \frac{C_2 \rho^2}{1 - \rho^2}, ~ \rho \in (0,1).
\end{equation}

With regard to Scenario (c), Lemma \ref{lemma-prob-bound} demonstrates that having more samples does not help in the case where \ref{condition-tail-pareto} holds; in fact, it actually makes the situation worse. This is expected as such a fat-tailed distribution has a population mean of infinity, which, intuitively, would make a sample mean entirely useless. Thus, considering \ref{condition-tail-pareto} for \ref{stop-mean-est} does not provide any additional value, and we will not pursue it further.

To summarize, under the general \ref{bcn} conditions and Scenario (b), Proposition \ref{prop-SC-2-detect-fn} shows that \ref{stop-mean-est} will be triggered in finite time for $N_j$ sufficiently large, and provides control over the false negative probability.

%% file: section/sc_majority_vote.tex
\subsection{Stopping Criterion by Majority Vote} \label{subsection-sc-maj-vote}

We now establish that \ref{stop-vote-ind} will be triggered with probability one and we derive bounds on the probability of a false negative. Just as we did in \S \ref{subsection-sc-grad-est}, we begin with a general result that will guide our specific analysis of \ref{stop-vote-ind}.

\begin{lemma}\label{lemma-prob-conc}
Let $F$ be a Bottou-Curtis-Nocedal nonconvex function (\S \ref{subsection-bcn-specification}). For $N \in \mathbb{N}$, let $\lbrace Z_1,\ldots,Z_N \rbrace$ be independent copies of $X$. Let $\epsilon > 0$ and $\delta \in (0,1)$ and define
\begin{equation}
\Delta = \delta - \Prb{ \norm{ \dot{f}(\theta,X) }_2 \leq \epsilon }.
\end{equation}
When $\Delta < 0$,
\begin{equation}
\Prb{ \frac{1}{N} \sum_{i=1}^{N} \1{ \norm{ \dot f (\theta, Z_i) }_2 \leq \epsilon } \geq \delta} \geq 1 - \exp\left( -2 N \Delta^2 \right).
\end{equation}
\end{lemma}
\begin{proof}
The proof leverages McDiarmid's inequality (see \S 3 of \cite{mcdiarmid1998}). Let the range of $X$ be denoted by $\mathcal{X}$. Let $z_1,\ldots,z_N \in \mathcal{X}$ and define
\begin{equation}
h(z_1,\ldots,z_N) = \frac{1}{N}\sum_{i=1}^N  \1{ \norm{ \dot f (\theta, z_i) }_2 \leq \epsilon }.
\end{equation}
Then, for any $j \in \lbrace 1,\ldots,N \rbrace$ and $z_1,\ldots,z_N,z_j' \in \mathcal{X}$,
\begin{equation}
| h(z_1,\ldots,z_N) - h(z_1,\ldots,z_{j}',\ldots,z_N)| \leq \frac{1}{N}.
\end{equation}
Since
\begin{equation}
\E{ h(Z_1,\ldots,Z_N) } = \Prb{ \norm{ \dot f (\theta, Z_1)}_2 \leq \epsilon},
\end{equation}
McDiarmid's inequality implies
\begin{align}
&\Prb{ h(Z_1,\ldots,Z_N) < \delta } \nonumber \\
&\quad = \Prb{ h(Z_1,\ldots,Z_N) - \E{ h(Z_1,\ldots,Z_N)} < \Delta } \\
&\quad \leq \exp( - 2N \Delta^2).
\end{align}
By computing the complement, the result follows. \qed
\end{proof}

If we compare Lemma \ref{lemma-prob-conc} to Lemma \ref{lemma-prob-bound}, we get two very different controls on the triggering probability: in the former, we have an exponential control in $N$, while in the latter we have sublinear control in $N$. As a result, increases in $N$ offer a dramatically greater benefit for \ref{stop-vote-ind} over \ref{stop-mean-est}, and this phenomenon underlies the famous theorem of machine learning referred to as ``classification is easier than regression'' \cite[Theorem 6.5]{devroye2013}. However, this benefit comes at a slight loss in generality: Lemma \ref{lemma-prob-conc} requires that $\mathbb{P}[ \Vert \dot f (\theta, X) \Vert_2 \leq \epsilon ]$ exceeds $\delta$, which may not hold for arbitrary values of $\epsilon > 0$, or for general \ref{bcn} functions such as the one induced by $f(\theta,X) = \theta X$ with $X$ as a Rademacher variable.

If we can allow for $\epsilon > 0$ to be large (e.g., $\epsilon \gg \sqrt {C_1}$), then we can account for general \ref{bcn} functions. However, if we cannot, then we must restrict our analysis to the situation where $\dot f(\theta,X)$ has some concentration near $\dot F (\theta)$, which is precisely induced in Scenarios (b) and (c). Therefore, we will restrict our attention to Scenarios (b) and (c), and we note that if $\epsilon$ is allowed to be large than Scenario (a) can be accounted for analogously to Scenario (b). We begin with proving that \ref{stop-vote-ind} will be triggered in finite time under Scenarios (b) and (c).

\begin{proposition}
Let $F$ be a \ref{bcn} function (\S \ref{subsection-bcn-specification}) with either $C_1 = 0$ in \ref{condition-noise-grad} or \ref{condition-tail-pareto}; and let $\lbrace \beta_k \rbrace$ be the iterates generated by Stochastic Gradient Descent satisfying \ref{prop-sympd} to \ref{prop-condnum} (\S \ref{subsection-sgd}). Let $\epsilon > 0$. Then \ref{stop-vote-ind} is triggered in finite time with probability one.
\end{proposition}
\begin{proof}
Let $\delta' \in (\bar \delta, 1)$, where $\bar \delta$ is defined in \ref{stop-vote-ind}. Now, by Corollary \ref{corollary-ST-converge}, there exists a finite random variable $J^*$ defined for each scenario such that for all $j \geq J^*$:
\begin{enumerate}
\item under Scenario (b), $\Vert \dot F ( \beta_{T_j} ) \Vert_2^2 < (1 - \delta') \epsilon^2 / (C_2 + 1)$; or
\item under Scenario (c), $\Vert \dot F ( \beta_{T_j} ) \Vert_2^2 < \epsilon (1 - \delta')^{1/\pi_2} / \pi_3$.
\end{enumerate}
Then, by Lemma \ref{lemma-prob-bound} (with $N=1$), under both scenarios, for $j \geq J^*$, with probability one, 
\begin{equation}
\condPrb{ \norm{ \dot f(\beta_{T_j}, Z_{ij}) }_2 \leq \epsilon }{\mathcal{F}_{T_j}} > \delta'.
\end{equation}
Therefore,
\begin{equation}
\delta_j - \condPrb{ \norm{ \dot f(\beta_{T_j}, Z_{ij}) }_2 \leq \epsilon }{\mathcal{F}_{T_j}} < \bar{\delta} - \delta' < 0.
\end{equation}
Apply Lemma \ref{lemma-prob-conc}, we conclude, for $j \geq J^*$, with probability one,
\begin{equation}
\condPrb{ \frac{1}{N_j} \sum_{i=1}^{N_j} \1{ \norm{ \dot f (\beta_{T_j}, Z_{ij})}_2 \leq \epsilon} \geq \delta_j }{\mathcal{F}_{T_j}} > 1 - \exp\left( - 2 (\delta' - \bar{\delta})^2 \right).
\end{equation}
Hence, the distribution of $J$, as defiend in \ref{stop-vote-ind}, for $j \geq J^*$ is controlled by a geometric distribution. We conclude $J$ is finite with probability one; that is, \ref{stop-vote-ind} is triggered in finite time with probability one.
\qed
\end{proof}

Our next step is to establish control over the probability of a false negative. Just as we did in Proposition \ref{prop-SC-2-detect-fn}, we will state these controls in terms of some parameters $\rho$ and $\gamma$; however, unlike in Proposition \ref{prop-SC-2-detect-fn}, we will need more control over the values of $\rho$. Specifically,
\begin{equation} \label{eqn-reduction-control}
\rho \in \left( 0, \sqrt{ \frac{ 1 - \bar{\delta} }{C_2 + 1} } \right) \quad\text{ or } \quad \rho \in \left( 0, (1 - \bar \delta )^{1/\pi_2}  \right),
\end{equation}
for Scenarios (b) or (c), respectively. The following result establishes control over the probability of a false negative for Scenario (b).

\begin{proposition}
Let $F$ be a \ref{bcn} function (\S \ref{subsection-bcn-specification}) with $C_1 = 0$ in \ref{condition-noise-grad}; and let $\lbrace \beta_k \rbrace$ be the iterates generated by Stochastic Gradient Descent satisfying \ref{prop-sympd} to \ref{prop-condnum} (\S \ref{subsection-sgd}). Let $\epsilon > 0$. Suppose $\rho$ satisfies \eqref{eqn-reduction-control} and $\gamma > 1$. If
\begin{equation}
N_j > \frac{\log(\gamma) }{2 \left(1 - \bar \delta - \rho^2/\gamma  \right)^2},
\end{equation}
then
\begin{equation}
\condPrb{ \frac{1}{N_j} \sum_{i=1}^{N_j} \1{ \norm{ \dot f (\beta_{T_j}, Z_{ij})}_2 \leq \epsilon} < \delta_j, ~ \norm{ \dot F (\beta_{T_j}) }_2 \leq \frac{\rho \epsilon}{\sqrt{\gamma}} }{\mathcal{F}_{T_j}} \leq \frac{1}{\gamma}.
\end{equation}
\end{proposition}
\begin{proof}
When $\Vert \dot {F}(\beta_{T_j}) \Vert_2 \leq \rho \epsilon / (\pi_3\gamma)$, by \eqref{eqn-reduction-control} and Lemma \ref{lemma-prob-bound},
\begin{equation}
\condPrb{ \norm{ \dot f (\beta_{T_j}, Z_{ij} ) }_2 \leq \epsilon }{\mathcal{F}_{T_j}} \geq 1 - \rho^2/\gamma > \bar{ \delta }.
\end{equation}
Therefore, 
\begin{equation}
\delta_j - \condPrb{ \norm{ \dot f (\beta_{T_j}, Z_{ij} ) }_2 \leq \epsilon }{\mathcal{F}_{T_j}} \leq \bar{\delta} + \rho^2/\gamma - 1 < 0.
\end{equation}
Hence, for the given choice of $N_j$,
\begin{align}
&\exp\left\lbrace - 2 N_j \left( \delta_j - \condPrb{ \norm{ \dot f (\beta_{T_j}, Z_{ij} ) }_2 \leq \epsilon }{\mathcal{F}_{T_j}} \right)^2 \right\rbrace \\
&\quad\leq \exp \left\lbrace -2N_j \left(\bar{\delta} + \rho^2/\gamma - 1  \right)^2 \right\rbrace \leq \frac{1}{\gamma}.
\end{align}

Noting that
\begin{equation}
\begin{aligned}
&\condPrb{ \frac{1}{N_j} \sum_{i=1}^{N_j} \1{ \norm{ \dot f (\beta_{T_j}, Z_{ij})}_2 \leq \epsilon} < \delta_j, ~ \norm{ \dot F (\beta_{T_j}) }_2 \leq \frac{\rho \epsilon}{\sqrt \gamma } }{\mathcal{F}_{T_j}} \\
&= \condPrb{ \frac{1}{N_j} \sum_{i=1}^{N_j} \1{ \norm{ \dot f (\beta_{T_j}, Z_{ij})}_2 \leq \epsilon} < \delta_j }{\mathcal{F}_{T_j}} \1{\norm{ \dot F (\beta_{T_j}) }_2 \leq \frac{\rho \epsilon}{\sqrt \gamma} },
\end{aligned}
\end{equation}
then, applying Lemma \ref{lemma-prob-conc} when the norm of $\dot{F}(\beta_{T_j})$ satisfies the hypothesized upper bound and $N_j$ satisfies the hypothesized lower bound, 
\begin{equation}
\condPrb{ \frac{1}{N_j} \sum_{i=1}^{N_j} \1{ \norm{ \dot f (\beta_{T_j}, Z_{ij})}_2 \leq \epsilon} < \delta_j }{\mathcal{F}_{T_j}} \leq \frac{1}{\gamma}.
\end{equation}
\qed
\end{proof}

We can derive a similar result for Scenario (c). 

\begin{proposition}
Let $F$ be a \ref{bcn} function (\S \ref{subsection-bcn-specification}) satisfying \ref{condition-tail-pareto}; and let $\lbrace \beta_k \rbrace$ be the iterates generated by Stochastic Gradient Descent satisfying \ref{prop-sympd} to \ref{prop-condnum} (\S \ref{subsection-sgd}). Let $\epsilon > 0$. Suppose $\rho$ satisfies \eqref{eqn-reduction-control} and $\gamma > 1$. If
\begin{equation}
N_j > \frac{\log(\gamma) }{2 \left(1 - \bar \delta - (\rho/\gamma)^{\pi_2}  \right)^2},
\end{equation}
then
\begin{equation}
\condPrb{ \frac{1}{N_j} \sum_{i=1}^{N_j} \1{ \norm{ \dot f (\beta_{T_j}, Z_{ij})}_2 \leq \epsilon} < \delta_j, ~ \norm{ \dot F (\beta_{T_j}) }_2 \leq \frac{\rho \epsilon}{\pi_3 \gamma} }{\mathcal{F}_{T_j}} \leq \frac{1}{\gamma}.
\end{equation}
\end{proposition}
\begin{proof}
When $\Vert \dot {F}(\beta_{T_j}) \Vert_2 \leq \rho \epsilon / (\pi_3\gamma)$, by \eqref{eqn-reduction-control} and Lemma \ref{lemma-prob-bound},
\begin{equation}
\condPrb{ \norm{ \dot f (\beta_{T_j}, Z_{ij} ) }_2 \leq \epsilon }{\mathcal{F}_{T_j}} \geq 1 - (\rho/\gamma)^{\pi_2} > \bar{ \delta }.
\end{equation}
Therefore, 
\begin{equation}
\delta_j - \condPrb{ \norm{ \dot f (\beta_{T_j}, Z_{ij} ) }_2 \leq \epsilon }{\mathcal{F}_{T_j}} \leq \bar{\delta} + (\rho/\gamma)^{\pi_2} - 1 < 0.
\end{equation}
Hence, for the given choice of $N_j$,
\begin{align}
&\exp\left\lbrace - 2 N_j \left( \delta_j - \condPrb{ \norm{ \dot f (\beta_{T_j}, Z_{ij} ) }_2 \leq \epsilon }{\mathcal{F}_{T_j}} \right)^2 \right\rbrace \\
&\quad\leq \exp \left\lbrace -2N_j \left(\bar{\delta} + (\rho/\gamma)^{\pi_2} - 1  \right)^2 \right\rbrace \leq \frac{1}{\gamma}.
\end{align}

Noting that
\begin{equation}
\begin{aligned}
&\condPrb{ \frac{1}{N_j} \sum_{i=1}^{N_j} \1{ \norm{ \dot f (\beta_{T_j}, Z_{ij})}_2 \leq \epsilon} < \delta_j, ~ \norm{ \dot F (\beta_{T_j}) }_2 \leq \frac{\rho \epsilon}{\pi_3 \gamma } }{\mathcal{F}_{T_j}} \\
&= \condPrb{ \frac{1}{N_j} \sum_{i=1}^{N_j} \1{ \norm{ \dot f (\beta_{T_j}, Z_{ij})}_2 \leq \epsilon} < \delta_j }{\mathcal{F}_{T_j}} \1{\norm{ \dot F (\beta_{T_j}) }_2 \leq \frac{\rho \epsilon}{\pi_3 \gamma} },
\end{aligned}
\end{equation}
then, applying Lemma \ref{lemma-prob-conc} when the norm of $\dot{F}(\beta_{T_j})$ satisfies the hypothesized upper bound and $N_j$ satisfies the hypothesized lower bound, 
\begin{equation}
\condPrb{ \frac{1}{N_j} \sum_{i=1}^{N_j} \1{ \norm{ \dot f (\beta_{T_j}, Z_{ij})}_2 \leq \epsilon} < \delta_j }{\mathcal{F}_{T_j}} \leq \frac{1}{\gamma}.
\end{equation}
\qed
\end{proof}

\begin{remark} The general case---Scenario (a)---can be addressed using these same techniques as long as $\epsilon$ is allowed to be sufficiently large so that we can provide a nontrivial upper bound on $\mathbb{P} [ \Vert \dot f (\theta, X ) \Vert_2 > \epsilon ]$. 
\end{remark}

To summarize, we have shown that under Scenarios (b) and (c), \ref{stop-vote-ind} will be triggered in finite time with probability one and the false negative probability can be controlled when $N_j$ is sufficiently large and $\Vert \dot {F}(\beta_{T_j}) \Vert_2$ is sufficiently small.

%% file: section/experiment.tex
We will discuss, in order, the background of the experiment, the experimental setup, the results of the experiment, and the interpretation of the results.

\subsection{Background}

For the setting of this experiment, we train a neural network to classify greyscale images of ten different types of clothing \cite{xiao2017}. We make use of 60,000 examples in this training, and the objective function is defined as the average sparse categorical cross entropy over the labels of the sixty thousand examples and the predicted labels as produced by a neural network that performs the following steps:
\begin{enumerate}
\item It flattens the images into vectors of dimension $784$;
\item Then it multiplies the vectors by a $128 \times 784$ matrix (to be estimated), adds a $128$-dimension vector (to be estimated) to the product, and applies a rectified linear unit function to each entry of the result of the affine transformation, which results in a $128$-dimension vector;
\item Then it multiplies the $128$-dimension vector by a $128 \times 128$ matrix (to be estimated), adds a $128$-dimension vector (to be estimated) to the product, and applies a rectified linear unit function to each entry of the result of the affine transformation, which results in a $128$-dimension vector;
\item Repeats the preceding step two more times with independent matrices and biases that must also be estimated;
\item Finally, it multiplies the $128$-dimension vector by a $10 \times 128$ matrix (to be estimated), adds a $10$-dimension vector (to be estimated), and uses the index of the maximum entry of this vector as its predicted label.
\end{enumerate}
The resulting parameter that is being optimized is of dimension $151,306$. 

The objective function is optimized using batch stochastic gradient descent with a scalar learning rate 
\begin{equation}
M_k = 0.2\frac{375000}{k + 375000} I,
\end{equation}
and is terminated after two hundred epochs (i.e., total passes through the data set). Note, at the beginning of the optimization, the learning rate is $0.2$ and by the end it is $0.1$, where $k$ represents the total number of batches processed in the $200$ epochs. At the end of each epoch, the parameter value is stored, the value of the objective function is calculated, and the norm of the gradient function is calculated. The value of the objective function and gradient function at the end of each epoch is plotted in Figure \ref{figure-objective-gradient-by-epoch}.

\begin{figure}[htb]
\centering
\subfloat[]{{\includegraphics[width=5cm]{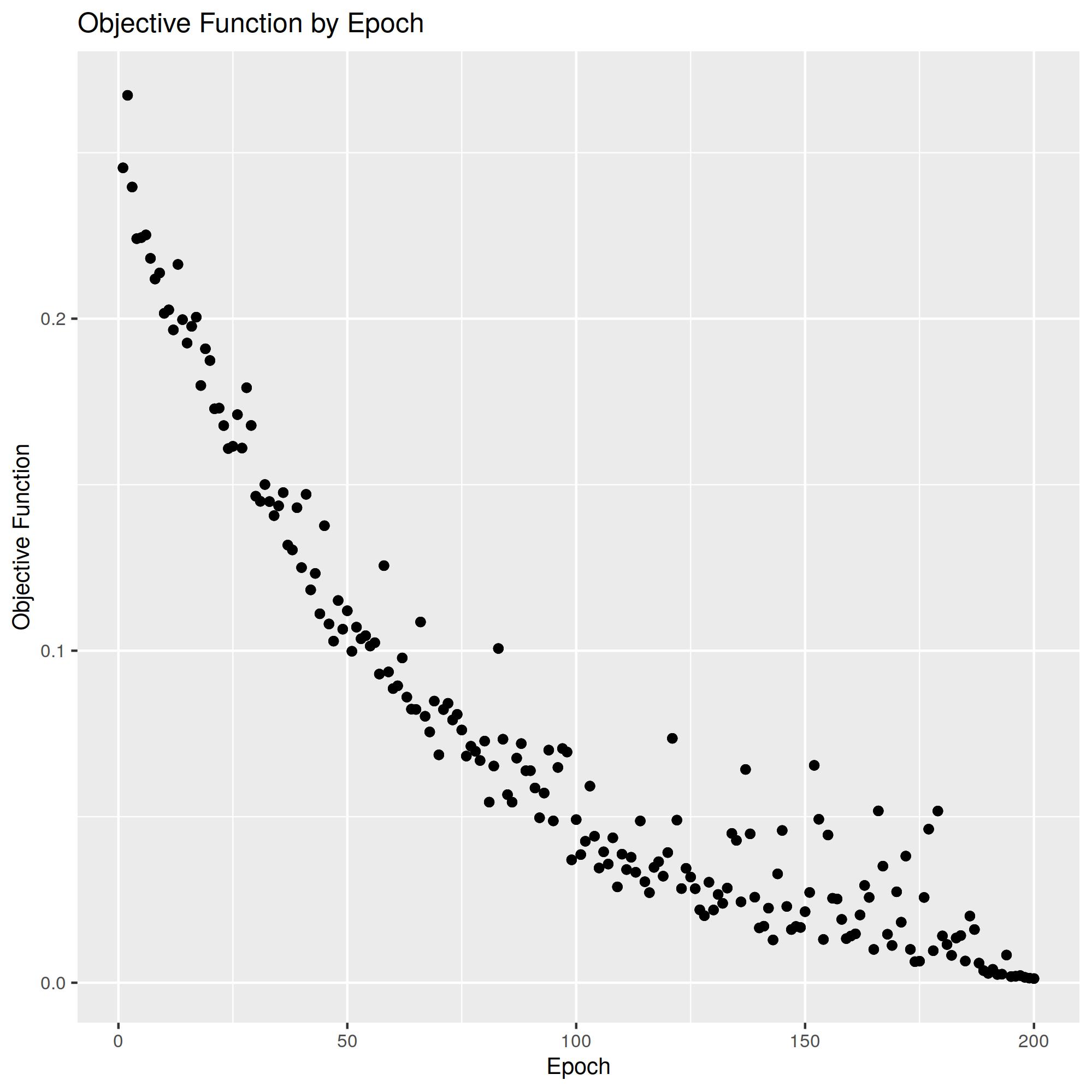} }}%
\qquad
\subfloat[]{{\includegraphics[width=5cm]{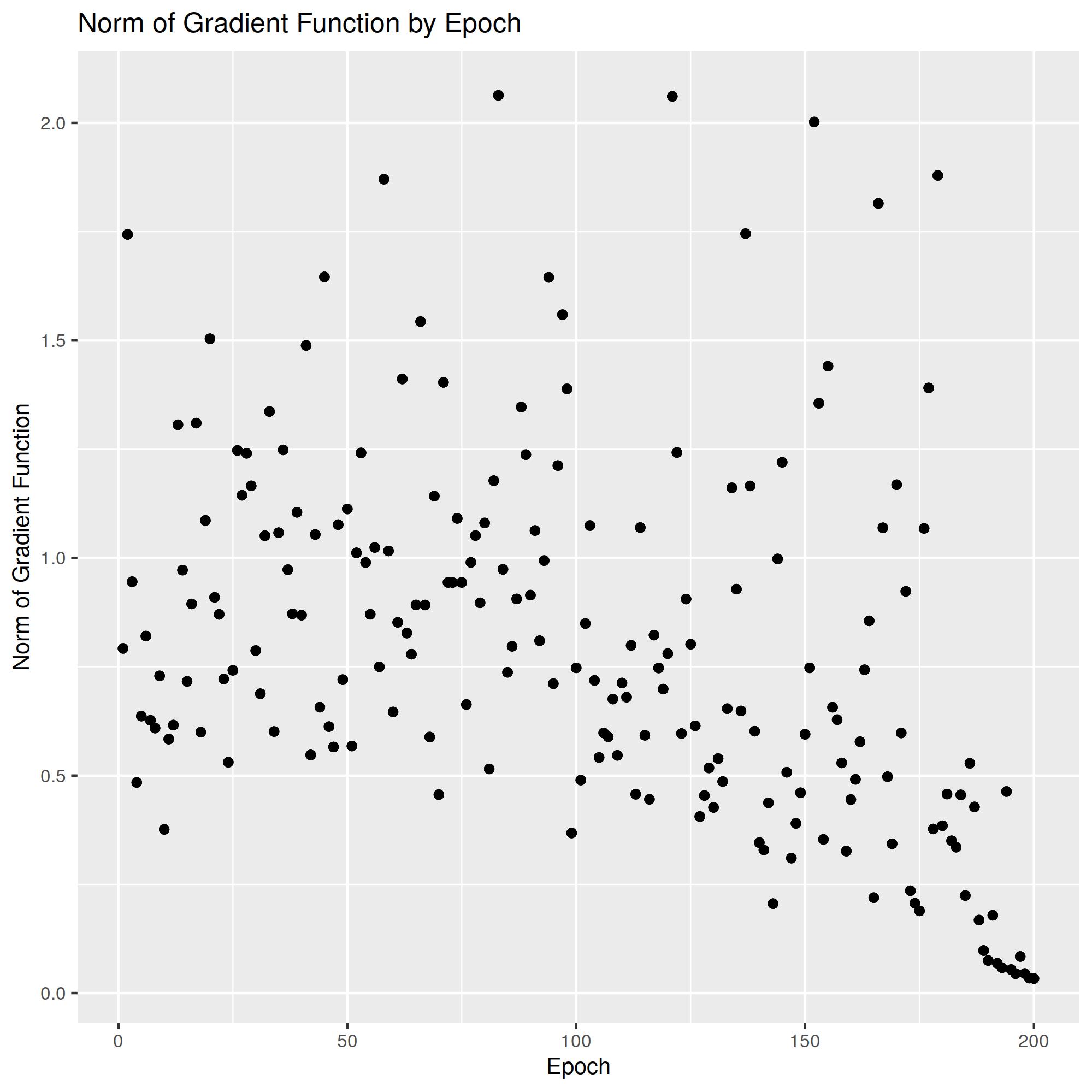} }}%
\caption{The value of the objective function (left) and norm of the gradient function (right) at the end of each of the two hundred epochs.}
\label{figure-objective-gradient-by-epoch}
\end{figure}

\subsection{Experimental Setup}

As an overview, at each of the recorded parameter values, the two stopping criteria are evaluated multiple times and the triggering rate (i.e., the number of times the stopping criterion is triggered divided by the total number of attempts) is recorded. Specifically, at each of the recorded parameter values---which occurs at the end of each epoch---, \ref{stop-mean-est} and \ref{stop-vote-ind} are evaluated one hundred times each using independent samples from the underlying data with the following specifications:
\begin{enumerate}
\item the threshold $\epsilon = 0.07$ (i.e., about twice the norm of the gradient at the final epoch);
\item the sample size $N_j$ takes value in $\lbrace 50,100,200\rbrace$; and,
\item for \ref{stop-vote-ind}, $\delta_j$ are kept constant and take value in $\lbrace 0.5, 0.6, 0.7, 0.8 \rbrace$. 
\end{enumerate}
In total, \ref{stop-mean-est} is evaluated on $600$ combinations ($200$ epochs, three distinct values of $N_j$), and \ref{stop-vote-ind} is evaluated on  $2,400$ combinations ($200$ epochs, three distinct values of $N_j$, four distinct values for $\delta_j$). Furthermore, each combination is evaluated independently $100$ times. At the end of each $100$ independent evaluations the triggering rate is calculated for each combination, which is the number of evaluated stopping criteria that are triggered divided by the total number of evaluated stopping criteria (i.e., $100$) for the given combination.  

\subsection{Results}

Figure \ref{figure-sc1-rate} shows the estimated triggering rate of \ref{stop-mean-est} as a function of the objective and as a function of the norm of the gradient. Note, Figure \ref{figure-sc1-rate}'s x-axes are reversed so as to roughly align with the increase in epochs; moreover, the points are jittered to allow for the zero values to be observable. Figure \ref{figure-sc1-rate} indicates that, as the objective function decreases (i.e., as we are presumably approach a solution) and as the norm of the gradient function decreases (i.e., as we are presumably approaching a stationary point), \ref{stop-mean-est} is more likely to be triggered. Moreover, Figure \ref{figure-sc1-rate} shows that the triggering rate is most variable for the case where $N_j = 50$ and least variable for the case of $N_j = 200$. 

\begin{figure}[htb]
\centering
\subfloat[]{{\includegraphics[width=5cm]{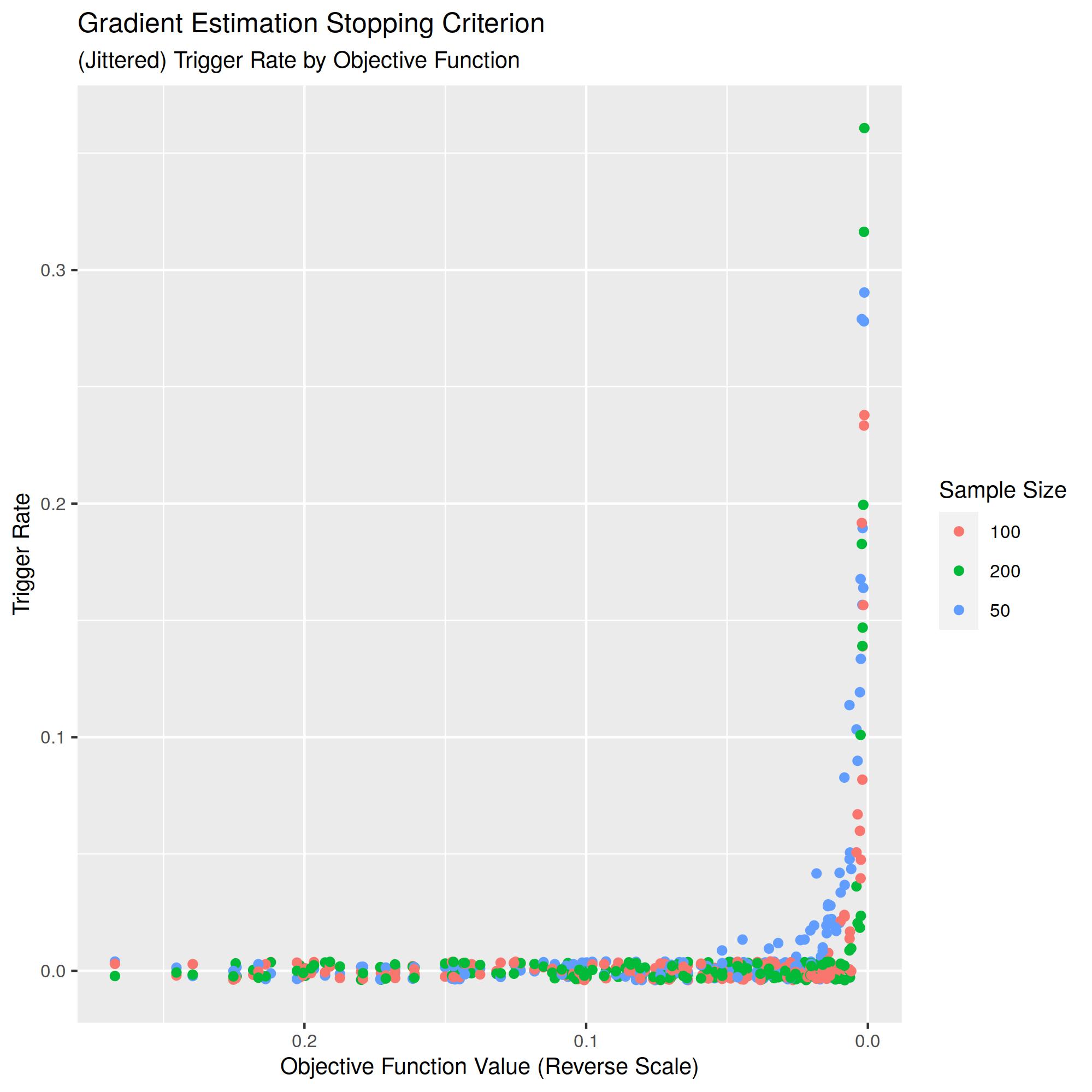} }}%
\qquad
\subfloat[]{{\includegraphics[width=5cm]{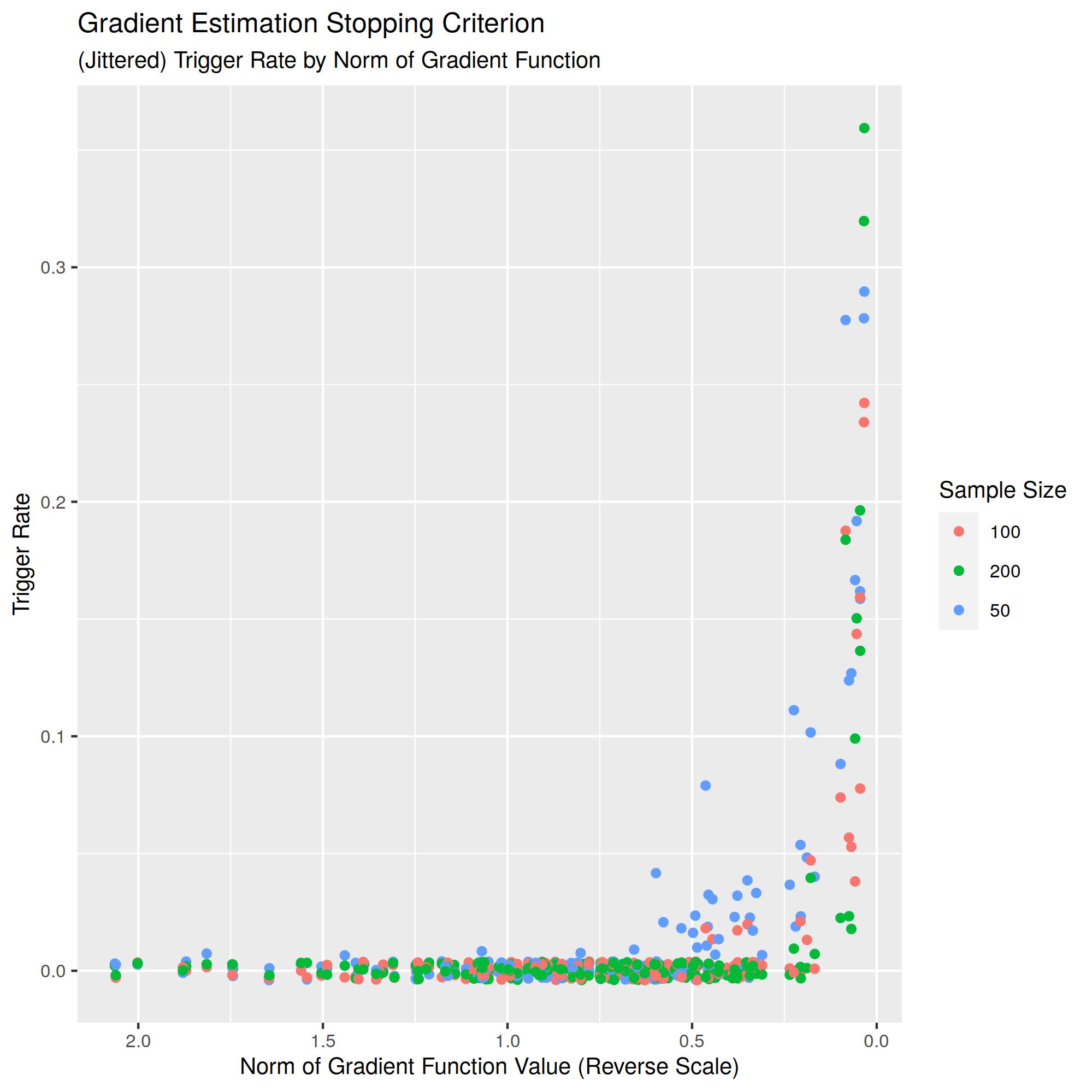} }}%
\caption{The triggering rate of \ref{stop-mean-est} by the value of the objective function (left) and norm of the gradient function (right). Note, the x-axis scales are reversed so as to roughly align with the increase in epochs.}
\label{figure-sc1-rate}
\end{figure}

Figure \ref{figure-sc2-rate} shows the estimated triggering rate of \ref{stop-vote-ind} as a function of the objective and as a function of the norm of the gradient. Note, Figure \ref{figure-sc1-rate}'s x-axes are reversed so as to roughly align with the increase in epochs; the points are jittered to allow for the zero values to be observable; and the panels each correspond to the different values of $\delta_j$ (i.e., the vote threshold). Figure \ref{figure-sc2-rate} indicates that, as the objective function decreases in value, \ref{stop-vote-ind} is more likely to be triggered---for the smaller vote thresholds, \ref{stop-vote-ind} is likely to be triggered at less optimal values, while, for larger vote thresholds, \ref{stop-vote-ind} is likely to be triggered close to the optimal value. On the other hand, Figure \ref{figure-sc2-rate} shows that, for all of the vote threshold values except for $\delta_j = 0.8$, \ref{stop-vote-ind} has no discernible pattern for when it is triggered. When $\delta_j = 0.8$, Figure \ref{figure-sc2-rate} shows that \ref{stop-vote-ind} is likely to be triggered only for small values of the gradient with the most variability occurring when $N_j = 50$ and least variability occurring  when $N_j = 200$. 

\begin{figure}[htb]
\centering
\subfloat[]{{\includegraphics[width=10cm]{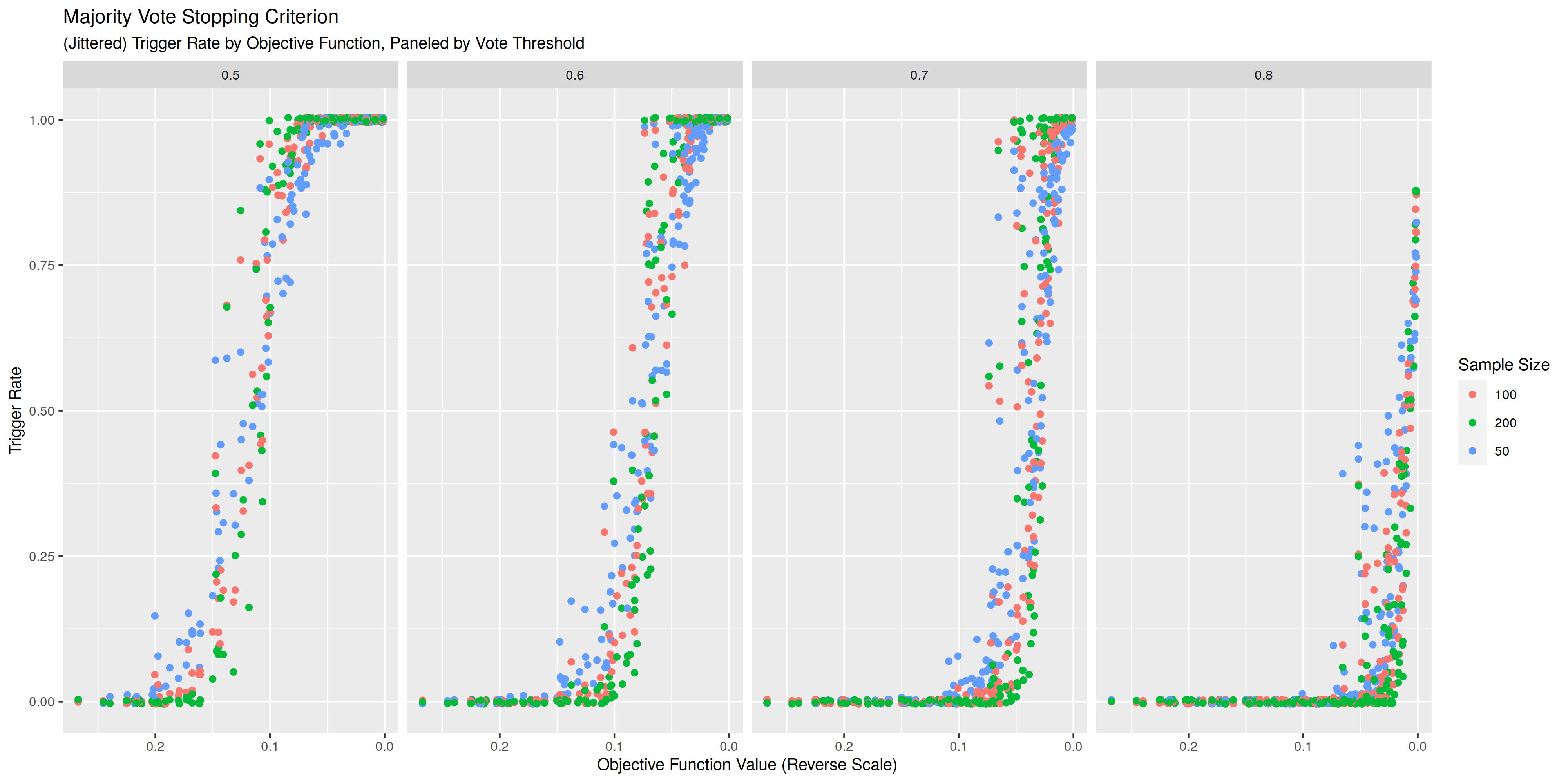} }} \\
\subfloat[]{{\includegraphics[width=10cm]{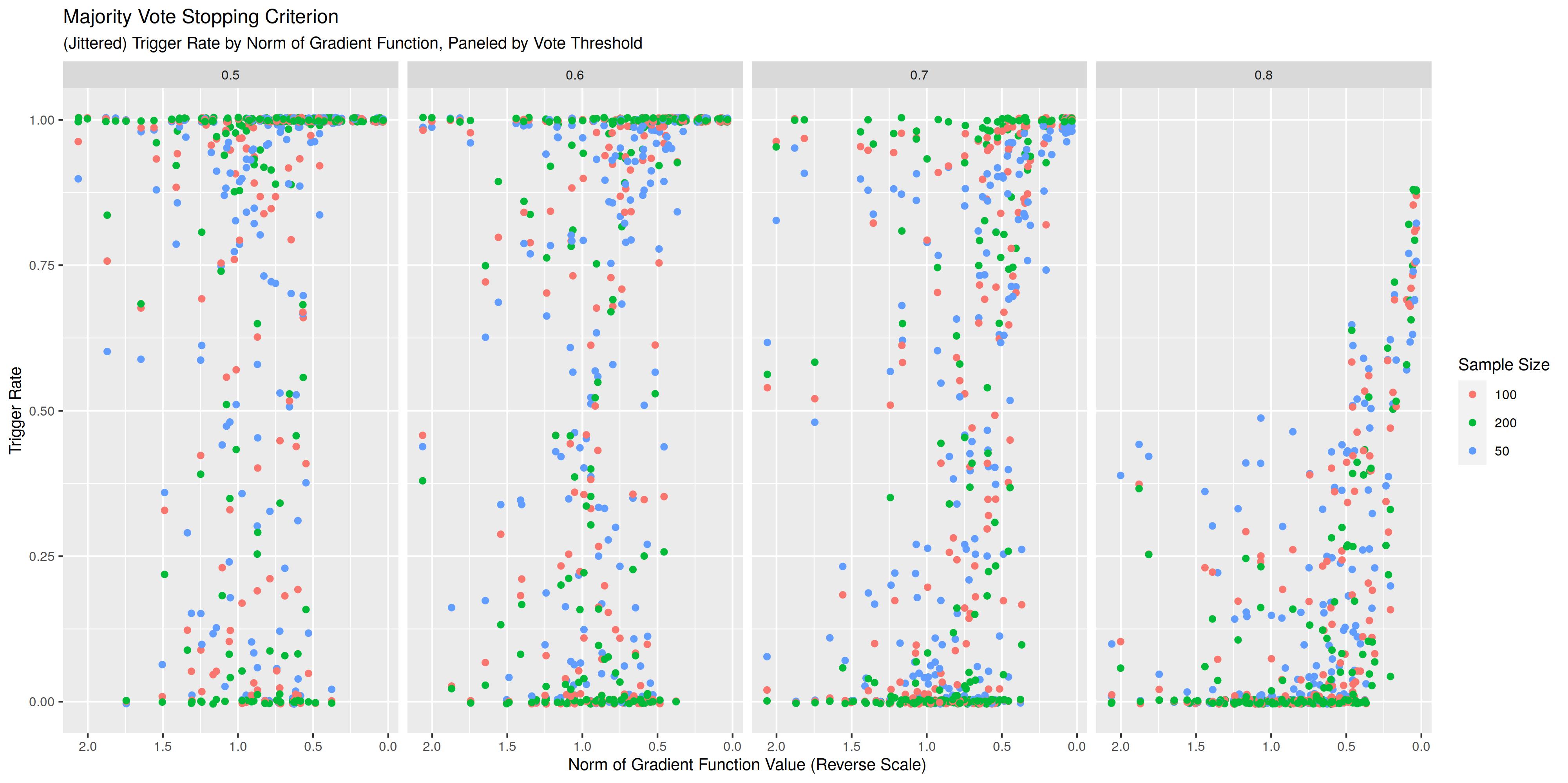} }}%
\caption{The triggering rate of \ref{stop-vote-ind} by the value of the objective function (left) and norm of the gradient function (right). Note, the x-axis scales are reversed so as to roughly align with the increase in epochs.}
\label{figure-sc2-rate}
\end{figure}

\subsection{Discussion}

As we would expect from the theory presented, \ref{stop-mean-est} correlates well with the decaying norm of the gradient and the probability of a false negative decays with increasing $N_j$. Moreover, \ref{stop-mean-est} is rather conservative: the probability of a false negative, even when $N_j = 200$ and the norm of the gradient is at its smallest, is $0.65$. Of course, $N_j = 200$ is less than $0.4\%$ of the total set of examples, and using more examples will better control the false negative probability, as we would expect.

Whereas \ref{stop-mean-est} behaved as we would expect, \ref{stop-vote-ind} is surprisingly more unpredictable. In particular, when the vote threshold was less than $0.8$, \ref{stop-vote-ind} was triggered for the larger values of the norm of the gradient. Unfortunately, for this particular problem, the distribution of the stochastic gradients are highly concentrated near zero with a handful that are far from zero. Therefore, when the vote threshold is low, this concentration of the stochastic gradients is likely to produce a false positive. However, when the vote threshold is increased, we see much more desirable behavior of \ref{stop-vote-ind}: it is being triggered only for smaller values of the norm of the gradient, and the false negative rate is closer to about $0.2$. Again, as expected from the theory, \ref{stop-vote-ind} is much more efficient with the observations in comparison to \ref{stop-mean-est}, but we must be careful and ensure that the value of $\delta_j$ is sufficiently large to prevent false positive signals.

%% file: section/conclusion.tex
In this work, our goal was to lay a rigorous foundation for stopping criteria for stochastic gradient descent (SGD) as applied to Bottou-Curtis-Nocedal (BCN) functions, which includes a broad class of convex and nonconvex functions. We started by developing a strong global convergence result for SGD on BCN functions, which generalizes previous results on the convergence of SGD on nonconvex functions. Then, we presented two stopping criteria and rigorously analyzed them.

This work has raised several questions that we enumerate below, and which we hope to address in future work. 
\begin{enumerate}
\item Given strong global convergence, what is the local rate of convergence to a stationary point? Are we even guaranteed to find a stationary point? Is this stationary point guaranteed to be a minimum? As evidenced by the many works cited, these issues are of great importance and can be answered more completely now that we have established strong global convergence.
\item Can something be said about dependent versions of the estimated stopping criteria? In some sense, dependent stopping criteria are ideal as they are the least wasteful stopping criteria. However, to develop such results, we need a maximal inequality over the norms of the iterates. 
\item For all of the stopping criteria, what are reasonable choices of $\lbrace T_j \rbrace$ and $\lbrace N_j \rbrace$? 
\item What are reasonable conditions to place on the lower tail probabilities (analogous to \ref{condition-tail-pareto}) and what are their implications for controlling the false positive probability of the stopping criteria studied in this work? 
\item Finally, is there a context in which \ref{condition-noise-obj} is more appropriate than \ref{condition-noise-grad}, and can the preceding results be developed in this context as well?
\end{enumerate}